\documentclass[12pt]{amsart}
\usepackage{amssymb}
\usepackage{amsbsy}
\usepackage{amscd}
\usepackage[mathscr]{eucal}
\usepackage{verbatim}
\usepackage[small]{diagrams}

\makeatletter
\def\cal{\mathcal}
\def\Bbb{\mathbb}
\def\frak{\mathfrak}

\newenvironment{pf*}[1]{\proof[#1]}{\endproof}
\makeatother

\newtheorem{Theorem}[equation]{Theorem}
\newtheorem{Corollary}[equation]{Corollary}
\newtheorem{Lemma}[equation]{Lemma}
\newtheorem{Proposition}[equation]{Proposition}

\newtheorem{Claim}{Claim}

\theoremstyle{definition}
\newtheorem{Definition}[equation]{Definition}

\newtheorem{Conjecture}[equation]{Conjecture}

\theoremstyle{remark}
\newtheorem{Remark}[equation]{Remark}

\numberwithin{equation}{section}

\newcommand{\thmref}[1]{Theorem~\ref{#1}}
\newcommand{\conref}[1]{Conjecture~\ref{#1}}

\newcommand{\lemref}[1]{Lemma~\ref{#1}}
\newcommand{\propref}[1]{Proposition~\ref{#1}}
\newcommand{\corref}[1]{Corollary~\ref{#1}}

\newcommand{\defref}[1]{Definition~\ref{#1}}
\newcommand{\remref}[1]{Remark~\ref{#1}}
\newcommand{\C}{{\Bbb C}}
\newcommand{\Z}{{\Bbb Z}}
\newcommand{\Q}{{\Bbb Q}}

\newcommand{\End}{\operatorname{End}}

\newcommand{\Ext}{\operatorname{Ext}}

\newcommand{\id}{\operatorname{id}}
\newcommand{\ve}{\varepsilon}
\def\m{\mathfrak m}

\newcommand{\OO}{\mathcal O}
\newcommand{\YY}{\mathcal Y}

\newcommand{\Pic}{\operatorname{Pic}}

\newcommand{\ch}{\operatorname{ch}}
\newcommand{\Wedge}{{\textstyle \bigwedge}}

\newcommand{\Tor}{\text{\rm Tor}}

\newcommand{\rk}{\mathop{{\rm rk}}}

\newcommand{\td}{\mathop{\text{\rm td}}}
\newcommand{\Coeff}{\mathop{\text{\rm Coeff}}}
\def\oo{{\cal O}}

\def\<{\langle}
\def\>{\rangle}
\def\[[{[\! [}
\def\]]{]\! ]}
\def\(({(\!(}
\def\)){)\!)}

\def\F{{\mathcal F}}

\def\E{{\mathcal E}}\def\hh{{\mathbb H}}

\newcommand{\vir}{{\mathop{\text{\rm vir}}}}
\newcommand{\perf}{{\mathop{\text{\rm perf}}}}
\newcommand{\mov}{{\mathop{\text{\rm mov}}}}
\newcommand{\fix}{{\mathop{\text{\rm fix}}}}

\newcommand{\EL}{{\mathop{\mathcal EL}}}
\newcommand{\Cc}{\mathfrak C}

\newcommand{\Eu}{\operatorname{Eu}}
\newcommand{\adj}{\operatorname{adj}}

\def\ZZ{\mathbb Z}
\def\FF{\mathcal F}
\def\Ppic{\mathop{\mathfrak Pic}}
\def\Gg{\mathbb G}

\def\Xx{\mathfrak X}
\def\XX{\mathcal X}
\def\Cc{\mathfrak C}
\def\eps{\varepsilon}

\begin{document}
\title[Riemann-Roch for virtually smooth schemes]{Riemann-Roch Theorems and Elliptic Genus \\ for  Virtually Smooth Schemes}
\author{Barbara Fantechi}
\address{SISSA,
Via Beirut 2/4,
34014 Trieste,
Italy}
\email{fantechi@sissa.it}
\author{Lothar G\"ottsche}
\address{ICTP, Strada Costiera 11, 
34014 Trieste, Italy}
\email{gottsche@ictp.it}
\begin{abstract} For a proper scheme $X$ with a fixed $1$-perfect obstruction theory $E^\bullet$, we define virtual versions of holomorphic Euler characteristic, $\chi_{-y}$-genus, and elliptic genus; they are deformation invariant, and extend the usual definition in the  smooth case. We prove virtual versions of the Grothendieck-Riemann-Roch and Hirzebruch-Riemann-Roch theorems. We show that the virtual $\chi_{-y}$-genus is a polynomial, and use this to define a virtual topological Euler characteristic. We prove that the virtual elliptic genus satisfies a Jacobi modularity property; we state and prove  a localization theorem in the toric equivariant case. We show how some of our results apply to moduli spaces of stable sheaves.\end{abstract}

\maketitle
\tableofcontents
\section{Introduction}

Let $X$ be a scheme which admits a global embedding in a smooth scheme, and $E^\bullet$ a $1$-perfect obstruction theory for $X$. One can view the pair $(X,E^\bullet)$ as being a {\em virtually smooth scheme} of expected (or virtual) dimension $d:=\rk E^\bullet$; indeed, many definitions for smooth schemes have been extended to this case, in particular, the pair has a virtual fundamental class $[X]^{vir}\in A_d(X)$ and a virtual structure sheaf ${\mathcal O}_X^{vir}\in K_0(X)$, which behave well under deformations of the pair.  

In this paper we want to extend to complete virtually smooth schemes other important notions:
 in particular, we define and study virtual versions of the holomorphic Euler characteristic for  elements $V\in K^0(X)$, and of the $\chi_{-y}$ genus and the elliptic genus. As a consequence, we can also define a virtual version of the topological Euler characteristic and of the signature. The virtual holomorphic Euler characteristic was already considered in  \cite{Lee}, although it is not given this name.
In this paper we will see that these invariants behave in a very similar way to their non-virtual counterparts for smooth complete schemes.
All of these invariants reduce to the usual ones if $X$ is smooth and $E^\bullet$ is the cotangent bundle, and they are deformation invariant.

The main results of the paper are a virtual version of the Theorems of Hirzebruch-Riemann-Roch and of Grothendieck-Riemann-Roch (in the latter, the target is supposed to be smooth, and not just virtually smooth). We also prove that the virtual $\chi_{-y}$-genus is actually a polynomial of degree $d$, and show that the virtual Euler number (defined as $\chi_{-1}(X)$) can be expressed as the degree of the virtual top Chern class. We show that as in the case of smooth varieties, the virtual elliptic genus of a virtual Calabi-Yau manifold is a weak Jacobi form. 
Finally, as an easy consequence of the virtual Riemann-Roch Theorem and 
the virtual localization of \cite{GP} we establish a localization formula for the  virtual holomorphic Euler characteristic, in case everything is equivariant under the action of a torus.

In the particular case where $X$ has lci singularities and $E^\bullet=L_X^\bullet$ the cotangent complex, we prove that the virtual topological Euler characteristic coincides with Fulton's Chern class and deduce that this is invariant under deformations for proper lci schemes.

This paper deals mostly with virtually smooth schemes and not with stacks, although it should be possible to generalize to the case of Deligne-Mumford stacks. We finish the paper by a partial generalization to the case of gerbes,
which allows to apply the results to moduli spaces of sheaves on surfaces
and on threefolds with effective anticanonical bundle.

The original motivation for this work comes mostly from moduli spaces of coherent sheaves on surfaces. 
In \cite{GNY},  $K$-theoretic Donaldson invariants are introduced as 
the holomorphic Euler characteristics of determinant bundles  on moduli spaces of stable coherent
sheaves on an algebraic surface $S$, and for surfaces with $p_g=0$ 
their wallcrossing behaviour 
is studied in the rank $2$ case, under assumptions that ensure that the moduli spaces are well-behaved. Using the virtual  Riemann-Roch theorem these assumptions can be removed, and many other results of \cite{Moc} that a priori only apply to the usual Donaldson invariants can be extended to the $K$-theoretic Donaldson invariants. In the forthcoming paper \cite{GMNY} this program is carried out. In particular is is easy to calculate the  wallcrossing
for the virtual Euler characteristic, and show that it is given by the same formula as in the case of so-called good walls (see \cite{Go}).
In \cite{DM} a  physical derivation of a wallcrossing formula for Euler numbers of moduli spaces of sheaves is given in a very general context, which in particular implies that the wallcrossing formula is the same in the virtual and in the 
non-virtual case.

The Euler numbers of moduli spaces of stable coherent sheaves on surfaces have been studied by many authors. In \cite{VW}, Vafa and Witten made predictions  about their generating functions,  in particular, they are supposed to be given by modular forms. This has been checked in a number of cases.
In general when these moduli spaces are very singular, there is to our knowledge no mathematical interpretation for the Euler numbers that figure in the predictions of \cite{VW}. We hope that our definition of the virtual Euler number will provide such an interpretation. 

The $\chi_{-y}$-genus and the elliptic genus are natural refinements of the 
Euler number. Thus it is natural to refine the virtual Euler characteristic to the virtual $\chi_{-y}$-genus and the virtual elliptic genus, and hope for their generating functions to have modularity properties. If the moduli spaces are smooth of the expected dimension, this has in many cases been shown.
The results in this paper are closely related to those obtained for $[0,1]$-manifolds by
Ciocan-Fontanine and Kapranov in  \cite{CFK}; there they prove the 
Hirzebruch-Riemann-Roch theorem and a localization formula in K-theory. As explained to us by Ciocan-Fontanine, from their results the Grothendieck-Riemann-Roch theorem for morphisms of $[0,1]$-manifolds easily follows under the same assumptions as in our paper.  They also construct a cobordism class associated to a $[0,1]$-manifold, which implies the possibility of introducing and studying genera for $[0,1]$-manifolds, such as the elliptic genus. \\
The language of $[0,1]$-manifolds and virtually smooth schemes are closely related as follows. If $\XX$ is a $[0,1]$-manifold, then 
$(\pi_0(\XX),\Omega^\bullet_{\XX|\pi_0(\XX)})$ is a virtually smooth scheme
by \cite[Prop.~3.2.4]{CFK}; on the other hand, it is expected that all virtually smooth moduli spaces arise in this way. This was proven by Ciocan-Fontanine and Kapranov in the following cases: for the Quot scheme and (in outline) the moduli stack of stable sheaves in \cite{CFK1}, and for the Hilbert scheme and the moduli stack of stable maps in \cite{CFK2}.\\
We thank I.~Ciocan-Fontanine and M.~Kapranov for showing us a preliminary version of the paper in June 2006 with 
the above mentioned material, except for the cobordism. One of the steps in our  proof of the virtual Riemann-Roch-Theorem is an adaptation of the corresponding argument in \cite{CFK}. 
Differently 
from Ciocan-Fontanine and Kapranov, our motivation for studying this problem, came from the study of $K$-theoretic Donaldson invariants,
and we also consider to some extent the stack version of 
the virtual Riemann-Roch theorem, as well as modular properties of the virtual elliptic genus. 

In the papers \cite{Jo1},\cite{Jo2} Joshua deals in great generality with the relation of the virtual fundamental class and the virtual structure sheaf
for Deligne-Mumford stacks with a perfect obstruction theory. Mochizuki informed us that he independently proved  the  virtual Hirzebruch-Riemann-Roch theorem, with applications to $K$-theoretic Donaldson invariants 
\cite{Moc2}.

\subsection{Acknowledgements} This work was partially supported by European Science Foundation Programme ``Methods of Integrable Systems, Geometry, Applied Mathematics" (MISGAM) Marie Curie RTN ``European Network in Geometry, Mathematical Physics and Applications" (ENIGMA) and by the Italian research grant PRIN 2006 ``Geometria delle variet\`a proiettive".

Part of this work was done while the authors were participating in the special year ``Moduli Spaces" at the Mittag-Leffler Institut, Djursholm, Sweden.

We want to thank Paolo Aluffi for useful discussions about Fulton's 
Chern class, 
Ionu\c{t} Ciocan-Fontanine for explanations about dg-manifolds and $[0,n]$-manifolds, and 
George Thompson for useful discussions about the $\chi_{-y}$-genus.

\section{Background material}

\subsection{Conventions} A scheme will be a separated scheme of finite type over an algebraically closed field of characteristic zero. We assume that all schemes under consideration admit a global embedding in a smooth scheme. \\If $S$ is a scheme, we denote by $A_*(S)$ the Chow group of $S$ {\em with rational coefficients}, and by $A^*(S)$ the Chow cohomology of $S$ (as defined in \cite[Def.~17.3]{Fu}), also with rational coefficients.\\ We often omit $i_*$ from the notation when $i:S\to S'$ is a closed embedding of schemes and $i_*$ is either the induced map $A_*(S)\to A_*(S')$ or $K_0(S)\to K_0(S')$.

\subsection{Grothendieck groups} Let $S$ be a scheme. We let $K^0(S)$ be the Grothendieck group generated by locally free sheaves, and $K_0(S)$ the Grothendieck group generated by coherent sheaves; we recall that $K^0(S)$ is naturally an algebra, contravariant under arbitrary morphisms, and $K_0(S)$ is a module over $K^0(S)$, covariant under proper morphisms.  Moreover, the natural homorphism $K^0(S)\to K_0(S)$ (induced by the inclusion of locally free sheaves inside coherent sheaves) is an isomorphism if $S$ is smooth. 

\subsection{Grothendieck groups and perfect complexes}\label{enfl} 
 The fact that a scheme $X$ can be  embedded as a closed subscheme in a smooth separated scheme implies that every coherent sheaf is a quotient of a locally free sheaf (it would be enough to assume irreducible, reduced and locally factorial instead of smooth: see \cite[exercise III.6.8]{Ha}). In other words, all schemes we consider have enough locally frees.

A complex $E^\bullet\in D^b(X)$ on an arbitrary scheme $X$ is called {\em perfect} if it is locally  isomorphic to a finite complex of locally free sheaves.   We write $D^b_{\perf}(X)$ for the full subcategory whose objects are the perfect complexes.\\
Since we assume that $X$ has enough locally frees, any perfect complex has a global resolution, i.e. a quasi-isomorphic complex which is globally a finite complex of locally frees.\\
One can therefore define for every object  $E^\bullet$ in $D^b_{\perf}(X)$ an element $[E^\bullet]\in K^0(X)$ defined to be equal to  $\sum_{i=m}^n(-1)^i[F^i]$ for $F^m\to\ldots\to F^n$ a global locally free resolution of $E^\bullet$; it is easy to show that the map $E^\bullet\mapsto [E^\bullet]$ is well-defined and behaves well with respect to quasi-isomorphisms and distinguished triangles.\\
In case one wants to extend the results of this paper to a more general situation, e.g. $X$ an arbitary scheme or an algebraic stack, care will have to be taken to assume that the relevant objects admit a global resolution.

\subsection{Todd and Chern classes}
Let $E$ be a rank $r$ vector bundle on a scheme $S$, and denote by $x_1,\ldots,x_r$ its Chern roots. The Chern character $\ch(E)$ and the Todd class $\td(E)$ are defined by $$\ch(V):=\sum_{i=1}^r e^{x_i}\ \text{and}\  \td(V):=\prod_{i=1}^r \frac{x_i}{1-e^{-x_i}}.$$ These extend naturally to a ring homomorphism  $\ch:K^0(S)\to A^*(S)$  and a group homomorphism $\td:(K^0(S),+)\to (A^*(S)^\times, \cdot)$ to the multiplicative group of units in the ring $A^*(S)$.\\
The morphism $\det$ associating to a rank $r$ vector bundle $E$ on $S$ its determinant $\det E:=\Wedge^rE\in \Pic X$ extends naturally to a group homomorphism $\det: K^0(X)\to \Pic(X)$.\\
The homomorphisms $\ch$, $\td$ and $\det$ commute with pullback via arbitrary morphisms.

\subsection{Perfect morphisms} \label{perfectmor} A morphism $f:X\to Y$ of schemes is {\em perfect}
\cite[Example 15.1.8]{Fu} if and
only if it factors as a closed embedding $i:X\to S$ such that $i_*(\oo_Z)\in D^b(S)$ is perfect and admits a global resolution, and a smooth morphism $p:S\to Y$. In particular every lci morphism of schemes admitting a closed embedding in smooth, separated schemes is perfect. In this case one can define a {\em Gysin homomorphism}
$$f^*:K_0(X)\to K_0(Y), f^*([\F])=\sum_{i}
(-1)^i\Tor^S_i(i_*\oo_Z,p^*\F).$$

Note in particular that the closed embedding of the zero locus of a regular section of a vector bundle is perfect, since we can always use the Koszul resolution.

\subsection{Riemann Roch for arbitrary schemes} \label{lcitau}
We will use the notations of \cite[Chapter 18]{Fu}.
In particular we use \cite[Theorem~18.2, Theorem~18.3]{Fu}.
For every scheme $S$ let $\tau_S:K_0(S)\to A_*(S)$ be the group homomorphism 
defined in \cite[Theorem~18.3]{Fu}. We recall in particular the following properties, which are taken almost verbatim from Theorem 18.3: \begin{enumerate}
\item module homomorphism: for any $V\in K^0(S)$ and any $\F\in K_0(S)$ one has $\tau_S(V\otimes \F)=\ch(V)\cap \tau_S(\F)$;
\item Todd: if $S$ is smooth, $\tau_S(\oo_S)=\td(T_S)\cap [S]$; hence for every $V\in K^0(S)$ one has $\tau_S(V\otimes \oo_S)=\ch(V)\cdot \td(T_S)\cap [S]$;
\item covariance: for every proper morphism $f:S\to S'$ one has $f_*\circ \tau_S=\tau_{S'}\circ f_*:K_0(S)\to A_*(S')$;
\item local complete intersection:
if $f:X\to Y$ is an lci morphism, and $\alpha\in K_0(Y)$, then $f^*(\tau_Y(\alpha))=(\td T_f)^{-1}\cap\tau_X(f^*\alpha)$.
\end{enumerate}

\subsection{Fulton's Chern class}
Let $X$ be a scheme and $i:X\to M$ a closed embedding in a smooth scheme. Fulton's Chern class of $X$ (we take the name from \cite{A}) is defined in 
\cite[Example 4.2.6]{Fu}  to be  $c_F(X)=c(T_M|_X)\cap s(X,M)\in A_*(X)$; it is shown there that $c_F(X)$ is independent of the choice of the embedding.  In \cite{A} $c_F(X)$ for hypersurfaces is related to the Schwarz-MacPherson Chern class $c_*(X)$, which has the property that 
 $\deg( c_*(X))=e(X)$.
 It is easy to see (\cite[Example 4.2.6]{Fu}), that for plane 
 curves $C$,
$\deg(c_F(C))=e(C')$ where $C'$ is a smooth plane curve of the same degree. Note that $C$ is lci and $C'$ is a smoothening of $C$. We will generalize this statement to arbitrary proper lci schemes in Theorem \ref{FultonChern}.

\section{Virtual Riemann-Roch theorems}

In this section we prove a virtual version of the Grothendieck-Riemann-Roch theorem for a proper morphism from a virtually smooth scheme to a smooth scheme. It would be interesting to have a more general version for proper morphisms of virtually smooth schemes, but at the moment we do not have that.

\subsection{Setup and notation}\label{setup} This setup will be fixed throughout the paper.
We will fix a scheme $X$ with a $1$-perfect obstruction theory $E^\bullet$; whenever needed, we also choose an explicit global resolution of $E^\bullet$ as a complex of vector bundles $[E^{-1}\to E^0]$, which exists by \ref{enfl}.

We denote by $[E_0\to E_1]$ the dual complex and by $d$ the expected dimension $d:=\rk E^\bullet=\rk E^0-\rk E^{-1}$. Recall that all schemes are assumed to be separated, of finite type over an algebraically closed field of characteristic $0$, and admitting a closed embedding in a smooth scheme.\\
Let $T_X^\vir\in K^0(X)$ be the class $[E_0]-[E_1]$. Note that (as explained in \ref{enfl}) $T_X^\vir$ only depends on $X$ and $E^\bullet$, and not on the particular resolution chosen.

\def\aff{\mathbb A}
\def\CC{\mathfrak C}
\def\Nn{\mathfrak N}
\def\Ee{\mathfrak E}

\subsection{Virtual fundamental class and structure sheaf}
We recall from \cite[Section 5]{BF} the definition of virtual fundamental class. Let $\Cc_X$ be the intrinsic normal cone of $X$; it is naturally a closed substack of $\Nn_X:=h^1/h^0((\tau_{\ge -1}L_X^\bullet)^\vee)$, the intrinsic normal sheaf. The map $\phi$ induces a closed embedding $\Nn_X\to \Ee:=h^1/h^0(E^\vee)$, and $\Ee=[E_1/E_0]$. 
Let $C(E)$ be inverse image of $\Cc_X$ in $E_1$ via the natural projection $E_1\to \Ee$; it  is a cone over $X$ of pure dimension equal to the rank of $E_0$. Let $s_0:X\to E_1$ be the zero section; $s_0$ is a closed regular embedding, hence following \cite[Def.~3.3]{Fu} we can denote by $s_0^*:A_*(E_1)\to A_*(X)$ the natural, degree $-\rk  (E_1)$ pullback map (or Gysin homomorphism). Then the virtual fundamental class $[X]^\vir$ is by definition equal to $s_0^*([C(E)])\in A_d(X)$.

We denote by $\oo_X^\vir\in K_0(X)$
the virtual structure sheaf of $X$, whose definition we now briefly recall.

\label{vireasy}
The virtual structure sheaf $\oo_X^\vir$ is equal to 
$$\sum_{i=0}^{\infty}(-1)^i\Tor^{E_1}_i(\oo_X,\oo_{C}).$$
Note that the sum is indeed finite; in fact, since $X$ is the zero locus of the tautological section $s$ of the bundle $\pi^*E_1$ (where $\pi:E_1\to X$ is the natural projection), we can give an explicit finite locally free resolution of $s_{0*}\oo_X$ on $E_1$  by the Koszul complex $(\Wedge^*(\pi^*E_1^\vee),s^\vee)$.
In other words, $s_0:X\to E_1$ is a perfect morphism in the sense of \ref{perfectmor} and
$$\oo_X^\vir=s_0^*([\oo_{C}]).$$
See \cite[Rem.~5.4]{BF}
for more details.

\subsection{A fundamental identity}
Let $[F^{-1}\to F^0]$ be a global resolution of $\tau_{\ge -1}(L_X)$,
i.e. a complex of coherent sheaves on $X$ with $F^0$ locally free of
rank $r$, with an isomorphism  $\psi:F_0\to \tau_{\ge -1}L_X^\bullet$.
Notice that such a global resolution is uniquely determined by a
morphism $\lambda:F^0\to \tau_{\ge -1}(L_X)$ such that
$h^0(\lambda):F^0\to \Omega_X$ is surjective.

We can use it to identify the intrinsic normal sheaf of $X$ with
$[F_1/F_0]$, and hence we get an induced cone $C(F)\subset F_1$ of pure
dimension $r$, the inverse image of the intrinsic normal cone inside
$[F_1/F_0]$. If $i:X\to M$ is a closed embedding in a smooth scheme,
we can choose $F^\bullet$ to be $[I/I^2\to \Omega_M|_X]$; we call it
the resolution induced by $i:X\to M$. If $\phi: E^\bullet\to \tau_{\ge
-1}L_X$ is an obstruction theory, then we can choose $F^0=E^0$ and
$F^{-1}=E^{-1}/\ker h^{-1}\phi$, with the induced map; we call it the
resolution induced by the obstruction theory. We denote by $p_F:C(F)\to X$ the natural projection.

\begin{Proposition}\label{goodprop} Let $F^\bullet $ be a presentation of $\tau_{\ge -1}(L_X)$. Let $p:C(F)\to X$ be the projection. Then $$
\tau_{C(F)}(\oo_{C(F)})=p_F^*(\td F_0)\cap [C(F)]\in A_*(C(F)).$$
\end{Proposition}
\begin{proof}
First step: it is enough to prove the proposition for one particular presentation. Indeed, given two presentations $F^\bullet$ and $G^\bullet$, we can compare either of them with the presentation $K^\bullet$ induced by $K^0:=F^0\oplus G^0\to \tau_{\ge -1}L_X$. As in \cite[Proposition 5.3]{BF}, the inclusion $F^0\to K^0$ induces a surjection $\bar\rho:K_1\to F_1$, and $C(K) =\bar\rho^{-1}(C(F))$. We let $\rho:C(K)\to C(F)$ be the restriction of $\bar\rho$; the map $\rho$ is part of  a natural exact sequence of cones (in the sense of \cite[Example 4.1.6]{Fu}) on $X$ $$
0\to G_0\to C(K)\to C(F)\to 0.$$
In particular $\rho$ is an affine bundle (in the sense of \cite[Section 1.9]{Fu}) with $T_\rho=p_K^*G_0$. Since it is an affine bundle, $\rho^*$ induces an isomorphism on Chow rings. So the statement holds for $F^\bullet$ if and only if  $$
\rho^*(\tau_{C(F)}(\oo_{C(F)}))=\rho^*(p_F^*(\td F_0)\cap [C(F)])\in A_*(C(K)).$$
Since $p_K=p_F\circ\rho$, and $\rho^*([C(F)]=(C[K])$, the right hand side is equal to $p_K^*(\td F_0)\cap [C(K)]$. By \ref{lcitau} applied to the smooth (hence lci) morphism $\rho$, the left hand side is equal to $\td(T_\rho)^{-1}\cap\tau_{C(K)}(\oo_{C(K)})$. 
So the equality holds for $F$ iff the equality $$
\td(T_\rho)^{-1}\cap\tau_{C(K)}(\oo_{C(K)})=p_K^*(\td F_0)\cap [C(K)]$$
holds in $A_*(C(K))$. Applying on both sides the invertible element $\td(T_\rho)=p_K^*(\td G_0)$ yields the equivalent formulation $$
\tau_{C(K)}(\oo_{C(K)})=(p_K^*(\td G_0)p_K^*(\td F_0))\cap [C(G)]$$
which is just the statement for $K$, since $K_0=F_0\oplus G_0$ and hence $\td K_0=\td F_0\cdot \td(G_0)$.\\
Second step: it is therefore enough to prove this in the case of
the resolution induced by a closed embedding $i:X\to M$  in a smooth
scheme (which exists by assumption). This is  proven in
in \cite[Lemma (4.3.2)]{CFK} under the additional assumption that $X$
and $M$ be quasiprojective, and we use a variation of their argument.
Let $\pi:\widetilde M\to \aff^1$ be the degeneration to the normal cone,
such that $\pi^{-1}(0)=C_{X/M}$ and $\widetilde
M_0:=\pi^{-1}(\aff^1_0)=M\times \aff^1_0$ (where we write $\aff^1_0$
for $\aff^1\setminus \{0\}$); let $q:\widetilde M\to M$ be the natural
morphism, composition of the blowup map $\widetilde M\to M\times\aff^1$
and the projection to the first factor.
Let $f:C_{X/M}\to \widetilde M$ be the natural closed embedding; $f$ is
regular, hence an lci morphism, and $T_f$ is $-[\oo_{C_{X/M}}]$, hence
$\td(T_f)=1$.
Let $\beta:=\tau_{\widetilde M}(\oo_{\widetilde M}-q^*\td(T_M))\cap [\widetilde
M]\in A_*(\widetilde M)$. By \ref{lcitau} applied to the
regular embedding $f$ with $\alpha=\oo_{\widetilde M}$, it is enough to
prove that $
f^*\beta=0$ in $A_*(C_{X/M})$ since $p=q\circ f$.
Let $j:\widetilde M_0\to \widetilde M$ be the (open) inclusion. Then
$j^*\beta=0$ since $\widetilde M_0$ is smooth and $\td(T_{\widetilde
M_0})=q^*\td(T_M)\cdot \pi^*\td(T_{\aff^1_0})=q^*\td(T_M)$. But the
argument in \cite[Section 10.1]{Fu} (paragraph starting ``if $T$ is a
curve'') show that from $j^*\beta=0$ we can deduce that $f^*\beta=0$,
thus completing the argument.
\end{proof}

In fact, assuming that we can extend Chapter 18 of \cite{Fu} to Artin
stacks, the Proposition takes the appealing form
$\tau_{\CC}(\oo_{\CC})=[\CC]$ where $\CC$ is the intrinsic normal
cone. 

\subsection{Main theorems}

\begin{Definition} 
The {\em virtual Todd genus} of $(X,E^\bullet)$ is defined to be $\td(T_X^\vir)$.
If $X$ is proper, then for any $V\in K^0(X)$, the {\em virtual holomorphic Euler characteristic}
is defined as
$$\chi^\vir(X,V):=\chi(X,V\otimes \oo_X^\vir).$$
\end{Definition}
\begin{Theorem}[virtual Grothendieck-Riemann-Roch]\label{GRR}
Let $Y$ be a smooth scheme and let $f:X\to Y$ be a proper morphism.
Let $V\in K^0(X)$. Then the following equality holds in $A_*(Y)$:
$$\ch(f_*(V\otimes \oo_X^\vir))\cdot\td(T_Y)\cap[Y]=
f_*(\ch(V)\cdot \td(T_X^\vir)\cap [X^\vir]).$$
\end{Theorem}
\begin{proof} The theorem follows by combining Lemma \ref{enough} with Lemma \ref{endpf} below.
\end{proof}

\begin{Corollary}[virtual Hirzebruch-Riemann-Roch]\label{HRR}
If $X$ is proper, and $V\in K^0(X)$, then 
$$\chi^\vir(X,V)=\int_{[X]^\vir}\ch(V) \td(T_X^\vir).$$
\end{Corollary}
\begin{proof}
This follows immediately by applying virtual Grothendieck-Riemann-Roch
to the projection of $X$ to a point.\end{proof}

We want to reduce in two steps \thmref{GRR} to a simpler statement.

\begin{Lemma}  To prove \thmref{GRR} it is enough to show that
\begin{equation}
\label{enn}
\tau_X(\oo_X^\vir)=\td(T_X^\vir)\cap [X]^\vir.
\end{equation}
\end{Lemma}

\begin{proof}
On the one hand  by the module property of \cite[Thm.~18.3]{Fu} we have
$$\tau_X(V\otimes \oo_X^\vir)=\ch(V)\cap\tau_X(\oo_X^\vir)
                             =\ch(V)\cdot\td(T_X^\vir)\cap [X]^\vir,$$
and thus
$$f_*(\tau_X(V\otimes \oo_X^\vir))=f_*(\ch(V)\cdot\td(T_X^\vir)\cap [X]^\vir).$$
On the other hand the covariance property of  \cite[Thm.~18.3]{Fu} gives
$$f_*(\tau_X(V\otimes \oo_X^\vir))=\tau_Y(f_*(V\otimes \oo_X^\vir)),$$
and because $Y$ is smooth we have
$$\tau_Y(f_*(V\otimes \oo_X^\vir))
                        =\ch(f_*(V\otimes \oo_X^\vir))\cdot\td(T_Y)\cap [Y].$$
and \thmref{GRR} follows.
\end{proof}

The formula \eqref{enn} is stated in \cite[Rem.~5.4]{BF}, however without proof.
It is proven in a different context in  \cite[Thm~1.5]{Jo1}. We prefer to
give a direct proof here since it is not clear to us how to relate
Joshua's results to what we need, and also since a direct proof is
very elementary.

\begin{Lemma}\label{enough} To prove \thmref{GRR} it is enough to show that 
$$s_0^*(\tau_{E_1}(\oo_{C}))=\td(E_0)\cap[X]^\vir.$$
\end{Lemma}
\begin{proof} 
Since $s_0:X\to E_1$ is a regular embedding, it is a local complete intersection morphism, 
with virtual tangent bundle $T_{s_0}=[-E_1]$.\\
By
\cite[Thm.~18.3(4)]{Fu}, we get
$$\tau_X(s_0^*(\oo_C))=\td(T_{s_0})\cdot s_0^*(\tau_{E_1}(\oo_C)).$$
In other words 
$$\tau_X(\oo_X^\vir)=\td(-E_1)\cdot s_0^*(\tau_{E_1}(\oo_{C})).$$
If $s_0^*(\tau_{E_1}(\oo_{C}))=\td(T_{E_1})\cap[X]^\vir,$
then we get by the above 
$\tau_X(\oo_X^\vir)=\td(-E_1)\cdot \td(E_0)\cap[X]^\vir,$ and we are done since
$$\td(T_X^\vir)=\td([E_0]-[E_1])=\td(E_0)\cdot \td(-E_1)$$ because $\td$ maps sums to products.\end{proof}

\begin{Lemma} Let $p:C\to X$ be the projection. Then $$
\tau_C(\oo_C)=p^*(\td E_0)\cap [C]\in A_*(C).$$
\end{Lemma}
\begin{proof}
This is a special case of Proposition \ref{goodprop}, when the resolution is induced by an obstruction theory.
\end{proof}

By \lemref{enough} we can finish our proof of \thmref{GRR} by showing the following
\begin{Lemma}\label{endpf} With the notation established so far,
$$s_0^*(\tau_{E_1}(\oo_C))=\td(E_0)\cap[X]^\vir.$$
\end{Lemma}
\begin{proof}
Let $j:C\to E_1$ be the embedding, $\pi:E_1\to X$ the projection, so that 
$p=\pi\circ j$.
By the covariance property \cite[Thm.~18.3(1)]{Fu}, the previous lemma and the projection formula we have 
\begin{align*}
\tau_{E_1}(\oo_{C})&=j_*(\tau_{C}(\oo_C))=j_*(p^*(\td(E_0))\cap [C])\\
&=j_*(j^*\pi^*(\td(E_0))\cap [C])=\pi^*(\td(E_0))\cap
j_*[C].
\end{align*}
Hence
\begin{align*}s_0^*(\tau_{E_1}(\oo_C))&=s_0^*(\pi^*(\td(E_0))\cap
j_*[C])\\ &=\td(E_0)\cap s_0^*(j_*([C]))=\td(E_0)\cap [X]^\vir.
\end{align*}
\end{proof}

\begin{Corollary}
If $X$ is proper and $d=0$, then 
$\chi^\vir(X,V)=\rk(V)\deg([X]^\vir).$\qed
\end{Corollary}

\begin{Corollary} If $X$ is proper, then
$\chi^\vir(X,\oo_X^\vir)=\int_{[X]^\vir} \td(T_X^\vir)$.\qed
\end{Corollary}

We finish this section by proving a weak virtual version of Serre duality.

\begin{Definition}
 The {\em virtual canonical (line) bundle}
 of $X$ is $K_X^\vir:=\det(E^0)\otimes \det(E^1)^\vee\in \Pic(X)$. Note that $K_X:=\det(T_X^\vir)^\vee$ only depends on the obstruction theory and not on the particular resolution chosen.
The {\em virtual canonical class} is $c_1(K_X^\vir)\in A^1(X)$.
If $c_1(K^\vir_X)=0$, we say that $X$ is a {\it virtual Calabi-Yau manifold}).\end{Definition}
In this paper the condition that $X$ is  a virtual Calabi-Yau manifold can always be replaced by the condition that $c_1(K^\vir_X)\cap [X]^\vir=0$ in $A_{d-1}(X)$.

\begin{Proposition}[weak virtual Serre duality]
If $X$ is proper and $V\in K^0(X)$, 
then
$\chi^\vir(X,V)=(-1)^d\chi^\vir(X,V^\vee\otimes K_X^\vir)$.
In particular if  $X$ is a virtual Calabi-Yau, then $\chi^\vir(X,V)=(-1)^d\chi^\vir(X,V^\vee)$.
\end{Proposition}
\begin{proof}
Let $n=\rk(E_0)$, $m=\rk(E_1)$, $d=n-m$  and let $x_1,\ldots,x_n$ be the Chern roots of $E_0$, 
$u_1,\ldots,u_m$ the Chern roots of $E_1$. We can assume that $V$ is a vector bundle on $X$. Let $v_1,\ldots v_r$ be its  Chern roots.
Then the virtual Riemann-Roch Theorem gives
\begin{align*}\chi^{\vir}(X,V)&=
\int_{[X]^\vir} \Big(\sum_{j=1}^r e^{v_j}\Big) 
\prod_{i=1}^n \frac{x_i}{1-e^{-x_i}} \prod_{k=1}^m \frac{1-e^{-u_k}} {u_k}
\\
\chi^{\vir}(X,V^\vee\otimes K_X^\vir)&=
\int_{[X]^\vir} \Big(\sum_{j=1}^r e^{-v_j}\Big) 
\prod_{i=1}^n \frac{x_i}{1-e^{-x_i}} \prod_{k=1}^m \frac{1-e^{-u_k}} {u_k}
\frac{\prod_{i=1}^n e^{-x_i}}{\prod_{k=1}^m e^{-u_k}}.
\end{align*}
By the identity $\frac{xe^{-x}}{1-e^{-x}}=\frac{-x}{1-e^{x}}$ in $\Q\[[x\]]$,
we see that the integrand for $\chi^{\vir}(X,V^\vee\otimes K_X^\vir)$
is obtained from that for $\chi^{\vir}(X,V)$ by replacing 
all $v_j$, $x_i$, $u_k$ by
$-v_j$, $-x_i$, $-u_k$ respectively.  This multiplies the part of degree $d$ by  $(-1)^d$. The result follows because $[X]^\vir\in A_d(X)$.
\end{proof}

\subsection{Deformation invariance}
\def\XX{\mathcal X}
\def\spec{\mathop{Spec}}

\begin{Definition} A {\em family of proper virtually smooth schemes}
is the datum of a proper morphism $\pi:\XX\to B$ of schemes with $B$ smooth, together with a $1$-perfect relative obstruction theory $E^\bullet$ for $\XX$ over $B$.\\ For every $b\in B$ a closed point, we will denote by $X_b$ the fiber $\pi^{-1}(b)$ and by $E^\bullet_b$ the induced obstruction theory for $X_b$. \\
For every $V\in K^0(\XX)$, we let $V_b:=V|_{X_b}\in K^0(X_b)$. Write $i_b:X_b\to \XX$ for the natural inclusion. In particular, we define $T^{vir}_{\XX/B}\in K^0(\XX)$ as the class associated to the complex $(E^\bullet)^\vee$; clearly $i_b^*T^{vir}_{\XX/B}=T^{vir}_{X_b}$.
\end{Definition}

Recall from \cite[Prop.~7.2]{BF} the following
\begin{Lemma} Let $\pi:\XX\to B$ be a family of proper virtually smooth schemes. Let $b:\spec K\to B$ be the morphism defined by the point $b$. Then $$
b^![\XX]^{vir}=[X_b]^{vir}.\qed$$
\end{Lemma}

We recall that the principle of conservation of number \cite[Proposition 10.2]{Fu} states that for any $\alpha\in A_{\dim B}(\XX)$, the degree of the cycle $\alpha_b:=i_b^!(\alpha)$ is locally constant in $b$.  The principle is in fact valid for arbitrary cycles in $A_*(\XX)$ if we use the convention that $\deg$ is defined on the $i$-th Chow group $A_i$ to be zero if $i\ne 0$. By using this principle, we immediately deduce the following Corollary.
\begin{Corollary}\label{definv} Let $\pi:\XX\to B$ be a proper family of virtually smooth schemes. For any $\gamma\in A^*(\XX)$, the number $$\int_{[X_b]^{vir}} i_b^*\gamma$$
is locally constant in $b$.
\end{Corollary}
\begin{proof} By definition, $$
\int_{[X_b]^{vir}} i_b^*\gamma=\deg \left(i_b^*\gamma\cap [X_b]^\vir\right)=\deg \left(i_b^*\gamma\cap i_b^![\XX]^\vir\right)=\deg i_b^!(\gamma\cap [\XX]^\vir).$$
Note that this number is zero if $\gamma\in A^e(\XX)$ and $e\ne d$, where $d$ is the virtual dimension of $X_b$.
\end{proof}

\begin{Definition} For any numerical object (e.g., a number or a function) which is defined in terms of a proper virtual smooth scheme $X$ and possibly of an element $V$ in $K^0(X)$, we say that it is {\em deformation invariant} if, for every family of proper virtually smooth schemes $\XX$ and every object $V\in K^0(\XX)$, the invariant associated to the virtually smooth scheme $X_b$ and the element $V_b$ is locally constant in $b$.
\end{Definition}

\begin{Theorem}\label{divc} Let $V\in K^0(\XX)$, and assume that $\pi$ is proper. Then $$
\chi^{vir}(X_b,V_b)$$
is locally constant in $b$. In other words, the virtual holomorphic Euler characteristic is deformation invariant. 
\end{Theorem}
\begin{proof} This is an immediate consequence of Corollary \ref{definv} and of virtual Hirzebruch-Riemann-Roch.
\end{proof}
  
\section{Virtual  $\chi_{-y}$-genus, Euler characteristics and signature}

In this section we introduce the virtual $p$-forms $\Omega^{p,\vir}_X$ on $X$ and define the virtual 
$\chi_{-y}$-genus $\chi^\vir_{-y}(X)$. A priori $\chi^\vir_{-y}(X)$  is just a formal power series in $y$. However we will prove that it is a  polynomial of degree $d$
in $y$ satisfying $\chi^\vir_{-y}(X)=y^d\chi_{-1/y}(X)$. The virtual Euler number  is then defined as $e^\vir(X):=\chi^\vir_{-1}(Y)$
and the virtual signature as $\sigma^\vir(X):=\chi^\vir_{1}(X)$.
We show a virtual version of the Hopf index theorem:
$e^\vir(X)=\int_{[X]^\vir} c_d(T_X^\vir)$. If $X$ is a proper local complete
intersection scheme with its natural obstruction theory, then 
$e^\vir(X)$ is the degree of Fulton's Chern class $c_F(X)$. 

\begin{Definition}
If $E$ is a vector bundle on $X$ of rank $r$ and $t$ a variable, we put
$$\Lambda_t E:=\sum_{i=0}^r [\Lambda^i E] t^i\in K^0(X)[t],\quad 
S_t(E):=\sum_{i\ge 0} [S^i E] t^i\in K^0(X)\[[t\]].$$
If $0 \to F\to G\to H\to 0$ is a short exact sequence of vector bundles on $X$, it easy to see that  
$\Lambda_t G=\Lambda_t F \times \Lambda_t H$.
Furthermore it is standard that 
$1/\Lambda_t(E)=S_{-t} E$ in $K^0(X)\[[t\]]$. Thus $\Lambda_t$ can be extended to a homomorphism $\Lambda_t:K^0(X)\to K^0(X)\[[t\]]$ by 
$\Lambda_t([E]-[F])=\Lambda_t(E) S_{-t}(F)$.
For $n\in \Z_{\ge 0}$ and $V\in K^0(X)$, we define $\Lambda^nV:=\Coeff_{t^n}\Lambda_tV\in K^0(X)$.

We define $\Omega_X^\vir:=(T_X^\vir)^\vee$.
The bundle of {\em virtual $n$-forms} on $X$ is $\Omega_X^{n,\vir}:=\Lambda^n(\Omega_X^\vir)\in K^0(X)$.
\end{Definition}

\begin{Definition} \label{chiy} 
Let $X$ be a proper scheme with a perfect obstruction theory and expected dimension $d$.
The {\em virtual $\chi_{-y}$-genus} of $X$ is defined by
$$\chi^\vir_{-y}(X):=\chi^\vir(X,\Lambda_{-y}\Omega^\vir_{X})=\sum_{p\ge 0} (-y)^p \chi^\vir(X,\Omega^{p,\vir}_X).$$
Let $V\in K^0(X)$. 
The {\em virtual $\chi_{-y}$-genus with coefficients in $V$} of $X$ is defined by
$$\chi^\vir_{-y}(X,V):=\chi^\vir(X,V\otimes \Lambda_{-y}\Omega^\vir_{X})=\sum_{p\ge 0} (-y)^p \chi^\vir(X,V\otimes \Omega^{p,\vir}_X).
$$
By definition $V\otimes \Lambda_{-y}(\Omega_X^\vir)\in K^0(X)\[[y\]]$, and thus $\chi^\vir_{-y}(X,V)\in \Z\[[y\]]$.

We will show below that $\chi^\vir_{-y}(X,V)\in \Z[y]$. Assuming this result for the moment, the {\em virtual Euler characteristic} of $X$ is defined as 
$e^\vir(X):=\chi^\vir_{-1}(X)$, and the {\em virtual signature} of $X$ as
$\sigma^\vir(X):=\chi^\vir_1(X)$.

Finally, for any partition $I$ of $d$, where $I=(i_1,\ldots,i_r)$ and $\sum_{k=1}^rk\cdot i_k=d$, we define the $I$-th virtual Chern number of $X$ to be $c_I(X):=\int_{[X]^\vir}\prod_{k=1}^rc_k^{i_k}(T_X^\vir)$. The virtual Chern numbers are deformation invariant by Lemma \ref{definv}.

Let $n=\rk(E_0)$, $m=\rk(E_1)$, $d=n-m$.
Let $x_1,\ldots,x_n$ be the Chern roots of $E_0$, $u_1,\ldots,u_n$ the Chern roots of $E_1$. 
We write $A^{>d}(X):=\sum_{l>d} A^l(X)$.
Let $A$ be the quotient of $A^*(X)$ by $A^{>d}(X)$.
We will denote classes in $A$ by the same letters as the  corresponding classes in $A^*(X)$.
By definition we have 
$$\ch(\Lambda_{-y}\Omega_X^\vir)=\frac{\prod_{i=1}^n (1-y e^{-x_i})}
{\prod_{j=1}^m(1-ye^{-u_j})}\in A\[[y\]],$$
where the right hand side is considered as element in $A\[[y\]]$ by the development
$$\frac{1}{\prod_{j=1}^m(1-ye^{-u_j})}=\prod_{j=1}^m\Big(\sum_{k\ge 0} y^k e^{-ku_j}\Big)\in A\[[y\]].$$
Let 
\begin{equation}\label{xxy}{\XX}_{-y}(X):=\ch(\Lambda_{-y}\Omega_X^\vir) \cdot \td(T_X^\vir)=
\prod_{i=1}^n\frac{x_i(1-y e^{-x_i})}{1-e^{-x_i}}\cdot 
\prod_{j=1}^m\frac{1-e^{-u_j}}{u_j(1-ye^{-u_j})}
\in A\[[y\]].
\end{equation}
By the virtual Riemann-Roch theorem we have
\begin{equation} \label{chiyRR}
\chi^\vir_{-y}(X)=\int_{[X]^\vir}{\XX}_{-y}(X), \ \chi^\vir_{-y}(X,V)=\int_{[X]^\vir}{\XX}_{-y}(X)\cdot \ch(V).
\end{equation}
\end{Definition}

\begin{Theorem}\label{Euler}
\begin{enumerate}
\item We have
$\chi^\vir_{-y}(X,V)\in \Z[y]$, and its degree in $y$ is at most $d$. 
\item $\chi_{-1}(X,V)=\rk(V) \int_{[X]^\vir} c_d(T_X^\vir)$.
\item In fact, we have ${\XX}_{-y}(X)\in A[y]$, and its degree 
in $y$ is at most $d$.
Furthermore we can write
${\XX}_{-y}(X)=\sum_{l=0}^d (1-y)^{d-l} {\XX}^l$ where
${\XX}^l=c_l(T_X^\vir)+b_l$ with $b_l\in A^{> l}(X)$.
\end{enumerate}
\end{Theorem}
\begin{proof} We
start by observing that it is enough to prove (3). Assume we know (3). Then \eqref{chiyRR} implies  that 
$\chi^\vir_{-y}(X,V)=\int_{[X]^\vir}{\XX}_{-y}(X)\cdot \ch(V)\in \Q[y]$
is a polynomial of degree at most $d$. By definition $\chi^\vir_{-y}(X,V)\in \Z\[[y\]]$, thus (1) follows.  
(3)  also gives 
$$\chi_{-1}(X,V)=\int_{[X]^\vir}{\XX}^0(X)\cdot \ch(V)= \rk(V)\int_{[X]^\vir}c_d(T_X^\vir),$$
which gives (2).  Thus we only have to show (3).
Let 
\begin{equation}\label{yyy}
{\cal Y}_z:=z^d\frac{\prod_{i=1}^n \big(x_i\frac{e^{-x_i}}{1-e^{-x_i}}+\frac{x_i}{z}\big)}{\prod_{j=1}^m\big( u_j\frac{e^{-u_j}}{1-e^{-u_j}}+\frac{u_j}{z}\big)}\in A[z].
\end{equation}
The right  hand side of \eqref{yyy} is seen to be an element of $A[z]$
as follows: for a variable $t$ write 
$\frac{te^{-t}}{1-e^{-t}}:=1+\sum_{k>0} a_k t^k\in \Q\[[t\]]$, which is obviously 
invertible in $\Q\[[t\]]$. Putting this into \eqref{yyy}, we get that 
\begin{equation}\label{yyyz}
{\cal Y}_z=z^d\frac{\prod_{i=1}^n \big(1+\sum_{k>0} a_k x_i^k+\frac{x_i}{z}\big)}{\prod_{j=1}^m \big(1+\sum_{k>0} a_k u_j^k+\frac{u_j}{z}\big)}
\in A^*(X)\((z^{-1}\)).
\end{equation}
Denote ${\cal Y}^{l}$ the coefficient of $z^{d-l}$ of ${\cal Y}_z$. Then we
see immediately from \eqref{yyyz} that ${\cal Y}^{l}\in A^{\ge l}(X)$ for all $l\ge 0$. In particular ${\cal Y}^{l}$ is zero  in $A$ for $d-l<0,$ and thus ${\cal Y}_z\in A[z]$. We also see that ${\cal Y}_z$ has at most degree $d$ in $z$. Furthermore \eqref{yyyz} also implies that the part of ${\cal Y}^{l}$ in $A^{l}(X)$ is the part in $A^l(X)$ of 
$\frac{\prod_{i=1}^n (1+x_i)}{\prod_{j=1}^m (1+u_j)}$, i.e. $c_l(T_X^\vir)$.
Thus in order to finish the proof we only have to see that
${\cal Y}_{1-y}={\XX}_{-y}(X)$ in $A\[[y\]]$. In $A\[[y\]]$ we have 
\begin{align*}{\cal Y}_{1-y}&=(1-y)^d\frac{\prod_{i=1}^n \big(x_i \frac{e^{-x_i}}{1-e^{-x_i}}+\frac{x_i}{1-y}\big)}{\prod_{j=1}^m \big(u_j \frac{e^{-u_j}}{1-e^{-u_j}}+\frac{u_j}{1-y}\big)}
=\frac{\prod_{i=1}^n x_i\big(1+(1-y) \frac{e^{-x_i}}{1-e^{-x_i}}\big)}
{\prod_{j=1}^m x_i\big(1+(1-y) \frac{e^{-u_j}}{1-e^{-u_j}}\big)}\\&=
\prod_{i=1}^n\frac{x_i(1-ye^{-x_i})}{1-e^{-x_i}}\prod_{j=1}^m
\frac{1-e^{-u_j}}{u_j(1-ye^{-u_j})}={\XX}_{-y}(X).
\end{align*}
\end{proof}

\begin{Corollary}[Hopf index theorem]\label{actually} The virtual Euler characteristic equals the top virtual Chern number $$
e^\vir(X)=c_d(X).$$
\end{Corollary}
\begin{proof} This is a special case of Theorem \ref{Euler}(2).
\end{proof}




\begin{Corollary}\label{sym}
Let $X$ be proper of expected dimension $d$, and $V\in K^0(X)$.
\begin{enumerate}
\item
$\chi^\vir_{-y}(X,V)=y^d \chi^\vir_{-1/y}(X,V^\vee).$
\item 
For all $p\ge 0$ we have $\chi^\vir(X,V\otimes \Omega^{p,\vir}_X)=
(-1)^d \chi^\vir(X,V^\vee \otimes \Omega^{d-p,\vir}_X)$, in particular 
$\chi^\vir(X,V\otimes \Omega^{p,\vir}_X)=0$ for $p>d$.
\end{enumerate}
\end{Corollary}
\begin{proof}
(1) Let again $n:=\rk(E_0)$, $m:=\rk(E_1)$, $d=n-m$.
It is well known that  $\Lambda^kF\otimes \det(F)^\vee\simeq
\Lambda^{r-k} F^\vee$ for  $F$  a  vector bundle of rank $r$.
Equivalently 
\begin{equation}\label{llamb}
(\Lambda_{-y}F)\cdot [\det(F)^\vee]=(-y)^r\Lambda_{-1/y}F^\vee,\hbox{  in }K^0(X)[y].
\end{equation}
By the weak virtual Serre duality we have 
\begin{equation}\label{chiT}
\chi^\vir_{-y}(X,V)=\chi^\vir(X,V\otimes\Lambda_{-y}\Omega_{X}^\vir)=
(-1)^d \chi^\vir(X,V^\vee \otimes K_X\otimes \Lambda_{-y} T_{X}^\vir).
\end{equation}
By \eqref{llamb}  we have in  $A^*(X)\[[y\]]$ the identity 
\begin{align*}
\ch(K_X\otimes \Lambda_{-y} T_{X}^\vir)&=
\ch\left(\frac{\det(E_0)^\vee\otimes  \Lambda_{-y}E_0}{\det(E_1)^\vee\otimes \Lambda_{-y}E_1}\right)\\
&=\ch\left(\frac{(-y)^n \Lambda_{-1/y}E^0}{(-y)^m\Lambda_{-1/y}E^{-1}}\right)=
(-y)^d \ch\big(\Lambda_{-1/y}\Omega_{X}^\vir\big).
\end{align*}
Thus \eqref{chiT} and the virtual Riemann-Roch theorem give
\begin{equation}\label{funeq}
\begin{split}
\chi^\vir_{-y}(X,V)&=(-1)^d\int_{[X]^\vir} \ch\big(V^\vee \otimes K_X\otimes \Lambda_{-y} T_{X}^\vir
\big)\cdot \td(T_X^\vir)\\& =y^d \int_{[X]^\vir}\ch\big(V^\vee \otimes \Lambda_{-1/y}\Omega_{X}^\vir\big)\cdot\td(T_X^\vir)=
y^d \chi^\vir_{-1/y}(X,V^\vee).
\end{split}
\end{equation} 
\eqref{funeq}  holds a priori in $\Q\[[y\]]$, but by \thmref{Euler} both sides are in 
in $\Z[y]$. Thus \eqref{funeq} holds in  $\Z[y]$. This proves (1). (2) is just a reformulation of (1).
\end{proof}

\begin{Remark} Thus we see that the virtual $\chi_{-y}$-genus, Euler number and signature have properties very similar to their non-virtual counterparts on  smooth projective varieties:
\begin{enumerate}
\item   $\chi^\vir_{-y}(X,V)$ is a polynomial of degree $d$ in $y$ with $\chi^\vir_{-y}(X,V)=y^d\chi^\vir_{-1/y}(X,V^\vee)$, in particular $\chi^\vir_{-y}(X)=y^d\chi^\vir_{-1/y}(X)$.
\item $e^\vir(X)=\chi^\vir_{-1}(X)=\int_{[X]^\vir}c_d(T_X^\vir)$.
\item If $d$ is odd, then $\sigma^\vir(X)=0$.
\item By definition $\chi^\vir_0(X,V)=\chi^\vir(X,V)$ and in particular 
$\chi^\vir_0(X)=\chi^\vir(X,\oo_X)$.
\end{enumerate}
\end{Remark}

\begin{Proposition} Let $X$ be a proper, virtually smooth scheme, and $V\in K^0(X)$. Then the virtual  $\chi_{-y}$ genus $\chi_{-y}^\vir(X,V)$ is deformation invariant. Hence, also the virtual Euler characteristic $e^\vir(X)$ and the virtual signature $\sigma^\vir(X)$ are deformation invariant.
\end{Proposition}
\begin{proof} This follows immediately from the definition and from Proposition \ref{divc}. For the virtual Chern numbers a different proof can be given  by combining Corollary \ref{actually} with the deformation invariance of the Chern numbers.\end{proof}



\subsection{The local complete intersection case}

We will say that the scheme $X$ has local complete intersection singularities, or just is lci (see \cite[Appendix B.7]{Fu}) if it admits a closed regular  embedding $i:X\to M$ in a smooth scheme. In this case $\tau_{\ge -1}L_X^\bullet=L_X^\bullet$ is a perfect complex, and thus a natural obstruction theory for $X$. Hence every lci scheme is naturally a virtually smooth scheme. The corresponding virtual fundamental class
is just $[X]$ and the virtual structure sheaf is $\oo_X$.

A {\em family of proper lci schemes} is an lci morphism $\pi:\XX\to B$ with $B$ smooth; again, the relative cotangent complex is also a relative obstruction theory, and hence $\pi:\XX\to B$ is also a family of proper virtually smooth schemes.

If $X$ is a proper lci scheme and $X_0$ is a proper smooth scheme, we say that $X_0$ is a {\em smoothening} of $X$ if there exists a family of proper lci schemes $\pi:\mathcal X\to B$ and closed points $b,b_0\in B$ such that $X_b$ is isomorphic to $X$ and $X_{b_0}$ is isomorphic to $b_0$.

\begin{Theorem}\label{FultonChern}
Let $X$ be a proper lci scheme with its natural obstruction theory: then $e^\vir(X)=\deg(c_F(X))$. Therefore $\deg(c_F(X))$ is invariant under lci deformations of $X$;
if $X$ admits a smoothing $X_0$,
then $
e^{vir}(X)=e(X_0)$.
\end{Theorem}
\begin{proof} It is of course enough to prove the first statement. 
By \cite[Ex.~4.2.6]{Fu} we have $c_F(X)=c(T_X^\vir)\cap [X]$, thus
if $d=\dim(X)$, then $\deg(c_F(X))=\int_X c_d(T_X^\vir)=e^\vir(X).$
\end{proof}
 
This  is easy to see in case $X$ is the zero scheme of a regular section of a vector bundle on a smooth proper variety  $M$. We thank P.~Aluffi for pointing this out to us.
 
 \begin{Remark}
More generally, all the virtual Chern numbers of a proper lci scheme are deformation invariants and coincide with the corresponding Chern numbers of a smoothing, when one exists.
\end{Remark}

\section{Virtual Elliptic genus}
Now we want to define and study a virtual version of the  Krichever-H\"ohn elliptic genus
\cite{Kr},\cite{Hoe}. 
The definition is completely analogous to the standard definition, 
we only replace at all instances $T_X$ by  $T_X^\vir$
and the holomorphic Euler characteristic by $\chi^\vir$.
Then we show that it has  similar properties to the elliptic genus of smooth projective varieties.
In particular, if $X$ is a virtual Calabi-Yau, i.e. the virtual canonical class of $X$ vanishes, 
then the elliptic genus is a meromorphic
 Jacobi form.

\begin{Definition}
As in the previous section let $A$ be the quotient of $A^*(X)$ by $A^{>d}(X)$ and denote by the same letter classes in $A^*(X)$ and in $A$.

For a vector bundle $F$ on $X$, we put 
$$\E(F)=\bigotimes_{n\ge 1}\big(\Lambda_{-yq^{n}} F^\vee
\otimes \Lambda_{-y^{-1}q^n} F\otimes S_{q^n} (F\oplus F^\vee)\big)\in 1+q\cdot K^0(X)[y,y^{-1}]\[[q\]].$$
Note that $\E$ defines a homomorphism from the
additive group $K^0(X)$ to the multiplicative group $1+q\cdot  K^0(X)[y,y^{-1}]\[[q\]]$.
For any vector bundle $F$ on $X$ we also put 
$$\EL(F;y,q):=y^{-\rk(F)/2}\ch(\Lambda_{-y}F^\vee) \cdot \ch(\E(F))\cdot \td(F)
\in A^*(X)[y^{-1/2},y^{1/2}]\[[q\]],$$
then  the map $F\mapsto \EL(F)$ extends to a homomorphism from the  additive group of $K^0(X)$
to the multiplicative group of $A^*(X)\((y^{1/2}\))\[[q\]].$
The {\em virtual elliptic genus} of $X$ is defined by
\begin{equation}\label{elldef}
Ell^\vir(X;y,q):=y^{-d/2}\chi^\vir_{-y}(X,\E(T_X^\vir))\in \Q\((y^{1/2}\))\[[q\]].
\end{equation}
For $V\in K^0(X)$ we also put 
$Ell^\vir(X,V;y,q):=y^{-d/2}\chi^\vir_{-y}(X,\E(T_X^\vir)\otimes V)$.
By our definitions and the virtual Riemann-Roch theorem we have 
$$Ell^\vir(X;y,q)=\int_{[X]^\vir} \EL(T^\vir_X;y,q),\quad Ell^\vir(X,V;y,q)=\int_{[X]^\vir} \EL(T^\vir_X;y,q)\cdot\ch(V),$$
and we see (in the notations of \eqref{xxy}) that 
$$\EL(T^\vir_X;y,q)=y^{-d/2}{\XX}_{-y}(X)\ch(\E(T_X^\vir )),$$
in particular, by \thmref{Euler} we see that $\EL(T^\vir_X;y,q)\in A[y^{1/2},y^{-1/2}]\[[q\]]$.
Finally for every  $k\in \Z_{\ge 0}$, $a\in A^k(X)$, we put
$Ell^\vir((X,a);y,q):=\int_{[X]^\vir} \EL(T^\vir_X;y,q)
\cdot a\in \Q[y^{1/2},y^{-1/2}]\[[q\]].$
From the definitions it is clear that
$Ell^\vir(X;y,q)|_{q=0}=y^{-d/2} \chi^\vir_{-y}(X).$

For $z\in \C$, $\tau\in \hh:=\big\{ \tau\in \C\bigm| \Im(\tau)>0\big\}$, 
we write 
\begin{align*}
\EL(F;z,\tau)&:=\EL(F;e^{2\pi i z},e^{2\pi i\tau}),\
Ell^\vir(X;z,\tau):=Ell^\vir(X;e^{2\pi i z},e^{2\pi i\tau}),\\
Ell^\vir(X,V;z,\tau)&:=Ell^\vir(X,V;e^{2\pi i z},e^{2\pi i\tau}),\ 
Ell^\vir((X,a);z,\tau):=Ell^\vir((X,a);e^{2\pi i z},e^{2\pi i\tau}).
\end{align*}
Let 
$$\theta(z,\tau):=q^{1/8}\frac{1}{i}(y^{1/2}-y^{-1/2})
\prod_{l=1}^\infty (1-q^l)(1-q^ly)(1-q^ly^{-1}),\qquad q=e^{2\pi i\tau}, \ y=e^{2\pi  i z}.$$ be the Jacobi theta function \cite[Chap.~V]{Cha}.
Let $f_1,\ldots,f_r$ be the Chern roots of $F$.
Then it is  shown e.g. in \cite[Prop.~3.1]{BL}, that 
\begin{equation}
\label{elltheta}
\EL(F;z,\tau):=\prod_{k=1}^r f_k\frac{\theta(\frac{f_k}{2\pi i}-z,\tau)}{\theta(\frac{f_k}{2\pi i},\tau)}.
\end{equation}
Let $B$ be the subring of $A\otimes \C$ generated by the Chern classes of 
$T_X^\vir$ and  a given element $a\in A^k(X)$. This is a finite dimensional $\C$-vector space.
It is easy to see that $\prod_{k=1}^r \frac{f_k}{\theta(\frac{f_k}{2\pi i},\tau)}$
defines  a holomorphic map from $\hh$ to the group of invertible elements $B^{\times}$ of $B$. Similarly $h:=\prod_{k=1}^r \theta(\frac{f_k}{2\pi i}-z,\tau)$ 
defines a holomorphic map from $\hh\times \C$ to $B$. 
Write $\Gamma:=\Z+\Z\tau$. Then we see
that the part of $h$ in $A^0(X)$ is given by $\theta(-z,\tau)$, which is nonzero
on $(\C\setminus\Gamma)\times \hh$. 
It follows that $\EL(T_X^\vir;z,\tau)$ is a holomorphic map 
$(\C\setminus\Gamma)\times \hh\to B$, and thus 
$Ell((X,a);z,\tau)$ and $Ell(X;z,\tau)$ are meromorphic functions from $\C\times \hh$ to $\C$.

\end{Definition}

Now we want to see that   that the virtual elliptic genus is a weak Jacobi form 
with character (see \cite[p.~104]{EZ}, for the definition in the case without character), as in the case of the usual elliptic genus of 
compact complex manifolds (see \cite{Hoe}, \cite{BL}).
We do not know whether (1) of \thmref{Jacobi} below was already known before in the case of compact complex manifolds.

\begin{Theorem} \label{Jacobi}
\begin{enumerate}
\item Fix $k\in \Z_{\ge 0}$, and let $a\in A^k(X)$ with $c_1(K^\vir_X)\cdot a\cap [X]^\vir=0$ in $A_{d-k-1}(X)$. 
Then $Ell^\vir((X,a);z,\tau)$ is a weak Jacobi form of weight $-k$ and index $d/2$ with character. Furthermore if $k>0$, then $Ell^\vir((X,a);0,\tau)=0.$
\item In particular if $X$ is a virtual Calabi-Yau manifold of expected dimension $d$, then 
$Ell^\vir(X;z,\tau)$ is a weak Jacobi form of weight $0$ and index $d/2$ with character. Furthermore $Ell^\vir(X;0,\tau)=e^\vir(X).$
\end{enumerate}
\end{Theorem}
\begin{proof}
First we want to show the  transformation properties.
By definition $Ell^\vir(X;z,\tau)=Ell^\vir((X,1);z,\tau)$, thus  it suffices to prove 
them for $Ell^\vir((X,a);z,\tau)$ for $a\in A^k(X)$.
The proof is a modification of that of \cite[Thm.~3.2]{BL}.
We have to show the equations
\begin{align}
\label{one}
Ell^\vir((X,a);z,\tau+1)&=Ell^\vir((X,a);z,\tau),\\
\label{two}Ell^\vir((X,a);z+1,\tau)&=(-1)^d Ell^\vir((X,a);z,\tau),\\
\label{three}Ell^\vir((X,a);z+\tau,\tau)&=(-1)^d e^{-\pi i d(\tau+2z)} Ell^\vir((X,a);z,\tau),\\
\label{four}Ell^\vir\left((X,a);\frac{z}{\tau},-\frac{1}{\tau}\right)&=\tau^{-k}e^{\frac{\pi i d z^2}{\tau}} Ell^\vir((X,a);z,\tau).
\end{align}
We have $T^\vir_X=[E_0-E_1]$, thus $\EL(T_X^\vir)=\EL(E_0)/\EL(E_1)$.
From \eqref{elltheta} and the standard identities (cf.\cite[V(1.4), V(1.5)]{Cha}) 
\begin{equation}\label{thtr}
\theta(z+1,\tau)=-\theta(z,\tau),\quad \theta(z+\tau,\tau)=-e^{-2\pi i z-\pi i\tau} \theta(z,\tau),\quad \theta(z,\tau+1)=\theta(z,\tau),
\end{equation}
we see that \eqref{one}, \eqref{two}, \eqref{three}  are satisfied for any vector bundle $F$, if we replace
$Ell^\vir((X,a);z,\tau)$ by $\EL(F;z,\tau)$ and $d$ by $\rk(F)$. As $\rk(E_0)-\rk(E_1)=d$, they also hold if we instead only replace $Ell^\vir((X,a);z,\tau)$
by $\EL(T_X^\vir;z,\tau)$. Thus they are also true for $Ell^\vir((X,a);z,\tau)=\int_{[X]^\vir} \EL(T_X^\vir;z,\tau)\cdot a$.

For a ring $R$, and $\alpha\in A^*(X)\otimes R$ denote by $[\alpha]_m$  the part  in $A^m(X)\otimes R$,
and for $G\in K^0(X)$, write $$\EL_\tau(G;z,\tau):=\sum_{m\ge 0}\tau^{-m} \left[\EL\left(G;z,\tau\right)\right]_m.$$
For  $F$ a vector bundle on $X$ of rank $r$ with Chern roots  $f_1,\ldots,f_r$, 
 \eqref{elltheta} gives
\begin{align*}\EL_\tau\left(F,\frac{z}{\tau},-\frac{1}{\tau}\right)&=
\prod_{l=1}^r\frac{f_l}{\tau}\frac{\theta(-\frac{z}{\tau}+\frac{f_l}{2\pi i \tau},-\frac{1}{\tau})}{\theta(\frac{f_l}{2\pi i \tau},-\frac{1}{\tau})}\\
&=\tau^{-r}\prod_{l=1}^r \Big(e^{\frac{-zf_l}{\tau}} f_l \frac{e^{\frac{\pi i z^2}{\tau}}
\theta(-z+\frac{f_l}{2\pi i},{\tau})}{\theta(\frac{f_l}{2\pi i},{\tau})}\Big)\\
&=\tau^{-r}e^{-\frac{zc_1(F)}{\tau}}e^{\frac{\pi i rz^2}{\tau}}\EL\left(F;z,\tau\right).
\end{align*}
Here the first and third line are obvious from the definitions, and the second line follows from the transformation properties  of $\theta$ by a two line  computation
(cf. \cite[eq. (9)]{BL}).
Thus 
\begin{align*}\EL_\tau\left(T^\vir_X,\frac{z}{\tau},-\frac{1}{\tau}\right)&=
\frac{\EL_\tau\left(E_0,\frac{z}{\tau},-\frac{1}{\tau}\right)}{\EL_\tau\left(E_1,\frac{z}{\tau},-\frac{1}{\tau}\right)}
=\tau^{-d}e^{\frac{zc_1(K_X^\vir)}{\tau}}e^{\frac{\pi i dz^2}{\tau}}\EL\left(T^\vir_X;z,\tau\right).
\end{align*}
By the definition of $\EL_\tau\left(T^\vir_X;z,\tau\right)$, we thus have 
$$\left[\EL\left(T^\vir_X,\frac{z}{\tau},-\frac{1}{\tau}\right)\right]_{d-k}=\tau^{-k}e^{\frac{\pi i dz^2}{\tau}}\left[e^{\frac{zc_1(K_X^\vir)}{\tau}}\EL\left(F;z,\tau\right)\right]_{d-k}.$$
As $a\cap [X]^\vir\in A_{d-k}(X)$, we thus get
\begin{align*}
Ell^\vir\left((X,a);\frac{z}{\tau},-\frac{1}{\tau}\right)
&=\tau^{-k}e^{\frac{\pi i dz^2}{\tau}}\int_{[X]^\vir}\Big[e^{\frac{zc_1(K_X^\vir)}{\tau}}\EL\left(F;z,\tau\right)\Big]_{d-k}\cdot a\\
&=\tau^{-k}e^{\frac{\pi i dz^2}{\tau}}\int_{[X]^\vir}\left[\EL\left(F;z,\tau\right)\right]_{d-k}\cdot a =\tau^{-k}e^{\frac{\pi i d z^2}{\tau}} Ell^\vir((X,a);z,\tau),
\end{align*}
where in the second line we have used that $K_X^\vir\cdot a\cap [X]^\vir=0$.
This shows that $Ell((X,a);z,\tau)$ has the  transformation properties of a Jacobi form of weight $-k$ and index $d/2$ with character.

Finally we show that $Ell((X,a);z,\tau)$ is regular on $\C\times \hh$ and at infinity and compute $Ell((X,a);0,\tau)$, $Ell(X;0,\tau)$.
Before the statement of \thmref{Jacobi} we saw that  
$$Ell^\vir((X,a);y,q):=\int_{[X]^\vir} {\XX}_y(X)\cdot  \ch(\E(T_X^\vir)
\cdot a\in \Q[y^{1/2},y^{-1/2}]\[[q\]],$$ i.e. $Ell^\vir((X,a);z,\tau)$ is holomorphic at infinity.
By \thmref{Euler}, we have ${\XX}_{y}(X)\in A[y]$ and ${\XX}_{-1}(X)=c_d(T_X^\vir)\in A$. By definition
$$\ch(\E(F))|_{y=1}=\prod_{n\ge 1} \ch\big(\Lambda_{-q^n} (F \oplus F^\vee)\big)
\ch(S_{q^{n}}(F\oplus F^\vee))=1$$ for all $F\in K^0(X)$.
Thus we get  for $a\in A^k(X)$  that $Ell^\vir((X,a);z,\tau)$ is holomorphic at $z=0$ and 
\begin{equation}\label{z0}Ell^\vir((X,a);0,\tau)=\int_{[X]^\vir} c_d(T_X^{\vir}) \cdot a=\begin{cases} 0 & k>0, \\ e^\vir(X)& a=1.\end{cases}
\end{equation}
Finally we can see that $Ell^\vir((X,a);z,\tau)$ is holomorphic  on $\C\times\hh$:
We had seen before that it is holomorphic on $(\C\setminus \Gamma)\times\hh$, and just saw  that it is holomorphic at $z=0$. Then \eqref{two} and \eqref{three} show that it is holomorphic at all $z\in \Gamma$.
\end{proof}

\begin{Proposition} Let $X$ be proper and $a\in A^*(X)$. Then the virtual elliptic genus $Ell^\vir((X,a);y,q)$ is deformation invariant. In particular if $X$ is a smoothable lci scheme with the natural osbtruction theory, then the virtual ellipitic genus $Ell^\vir(X;y,q)$ coincides with the elliptic genus of a smoothening.
\end{Proposition}
\begin{proof} This is a direct consequence of the definition and of Lemma \ref{definv}.
\end{proof}

\section{Virtual localization}
In this section let $X$ be a proper scheme over $\C$ with a $\C^*$-action and an equivariant $1$-perfect obstruction theory. We also assume that $X$ admits
an equivariant embedding into a nonsingular variety.
We denote by $K^0_{\C^*}(X)$, the Grothendieck group of equivariant vector bundles on $X$.
We  combine the virtual Riemann-Roch formula with the virtual localization of \cite{GP}
to obtain a localization formula expressing $\chi^\vir(X,V)$, in terms of the equivariant virtual holomorphic Euler characteristics on fixpoint schemes. 

Let $Z$ be a scheme on which $\C^*$ acts trivially.
Let $\ve$ be a variable. Let $B$ be a vector bundle on $Z$ with a $\C^*$-action. 
Then $B$ decomposes as a finite direct sum 
\begin{equation}\label{decomp} B=\bigoplus_{k\in \Z}B^k
\end{equation}
of $\C^*$-eigenbundles  $B^k$ on which $t\in\C^*$ acts by  $t^k$.
We identify 
$B$ with $\sum_{k} B^k e^{k\ve}\in K^0(X)\[[\ve\]].$ 
 This identifies $K^0_{\C^*}(X)$ with a subring of  $K^0(X)\[[\ve\]]$.
Now let again $B$ be a $\C^*$-equivariant vector bundle on $Z$ with decomposition \eqref{decomp} into eigenspaces.
We 
denote 
$$B^\fix:=B^0,\qquad B^\mov:=\oplus_{k\ne 0} B^k.$$
We put 
$\Lambda_{-1} B:=\sum_{i\ge 0} (-1)^i \Lambda^iB.$
If $B=B^\mov$, then $\Lambda_{-1}B$ is  invertible  in $K^0(Z)\((\ve\))$.
Thus if $C\in K^0_{\C^*}(Z)$ is of the form $C=A-B$, with
$A=A^\mov$, $B=B^\mov$, then $\Lambda_{-1}C:=\Lambda_{-1}A
/\Lambda_{-1}B$ is an invertible element in $K^0(Z)\((\ve\))$.

Now assume that $Z$ is proper and  has a $1$-perfect obstruction theory. Let $\oo_Z^\vir$ be the corresponding 
virtual structure sheaf. Let 
$p_*^\vir:=K^0(Z)\((\ve\))\to \Q\((\ve\))$ be the $\Q\((\ve\))$-linear extension of 
$\chi^\vir(Z,\bullet):K^0(X)\to\Z$.

We also recall some basic facts about equivariant Chow groups. 
Let $A^*_{\C^*}(Z)$ be the  equivariant Chow ring, this can be canonically identified with $A^*(Z)[\ve]$.
We extend the Chern character $\ch:K^0(Z)\to A^*(Z)$, by 
$\Q\((\ve\))$-linearity to $\ch:K^0(Z)\((\ve\))\to A^*(Z)\((\ve\))$.
With our identification of $K^0_{\C^*}(Z)$ with a subring of $K^0(Z)\[[\ve\]]$,
the restriction to $K^0_{\C^*}(Z)$ is the equivariant Chern character.
For $V\in K^0_{\C^*}(Z)$ let $\Eu(V)\in A^*_{\C^*}(Z)$ be the equivariant Euler class, and $\td(V)$ the equivariant Todd genus. By definition, we have 
$\ch(\Lambda_{-1} V^\vee)=\Eu(V)/\td(V)$.
Let $p_*:A_*^{\C^*}(Z)=A_*(Z)[\ve]\to \Q[\ve]$ be the equivariant pushforward to a point; it is 
$\Q[\ve]$--linear and we denote by the same letter its  $\Q\((\ve\))$--linear extension.
As the action of $\C^*$ on $Z$ is trivial, $\td(T_Z^\vir)$ and $[Z]^\vir$ are
$\C^*$-invariant. Thus $\alpha\mapsto p_*( \td(T_Z^\vir)\cdot \alpha\cap [Z]^\vir)$ is $\Q\((\ve\))$-linear. As $\int_{[Z]^\vir} \ch(V) \td(T_Z^\vir)=\chi^\vir(Z,V)$ for $V\in K^0(Z)$, it follows that 
\begin{equation}\label{intt}
p_*^\vir(V)=p_*(\ch(V) \td(T_Z^\vir)\cap [Z]^\vir) ,\quad \hbox{for }V\in K^0(Z)
\((\ve\)).
\end{equation}

We briefly recall the  setup of \cite{GP}.
We assume that $X$ admits an equivariant global embedding into a nonsingular scheme $Y$ with $\C^*$ action.
Let $I$ the ideal sheaf of $X$ in $Y$, assume that 
$\phi:E^\bullet \to [I/I^2\to \Omega_Y]$ is a map of complexes.
Assume that the action of $\C^*$ lifts to $E^\bullet$ and $\phi$ is equivariant.
Then \cite{GP} define an equivariant fundamental class $[X]^\vir$ in 
the equivariant Chow group $A_d^{\C^*}(X)$.
Let $X^f$ be the maximal $\C^*$-fixed closed subscheme of $X$.
For nonsingular $Y$, $Y^f$ is the nonsingular set-theoretic fixpoint locus, and $X^f$ is the scheme-theoretic intersection
$X^f=X\cap Y_f$.
Let  $Y^f:=\bigcup_{i\in S} Y_i$ be the decomposition into irreducible components and $X_i=X\cap Y_i$.
The $X_i$ are possibly reducible.
\cite{GP} define a canonical  obstruction theory on $X_i$. Let $[X_i]^\vir$ be the corresponding virtual fundamental class and
$\oo_{X_i}^\vir$ the corresponding virtual structure sheaf. 
The {\em virtual normal bundle} $N_i^\vir$ of $X_i$ is defined by
$N_i^\vir:=(T_X^\vir|_{X_i})^\mov$.

\begin{Proposition}[weak K-theoretic localization]\label{loc}
Let $V\in K_\C(X)$ and let $\widetilde V\in K^0_\C(X)$ be an equivariant lift of $V$. Denote by $\widetilde V_i$ the restriction of $\widetilde V$ to $X_i$ and
$p_i:X_i\to pt$ the projection.
Put 
$$\chi^\vir(X,\widetilde V,\ve):=\sum_{i} p_{i*}^\vir\big(\widetilde V_{i}/\Lambda_{-1}(N_{i}^\vir)^\vee)\big).$$ 
Then $\chi^\vir(X,\widetilde V,\ve)\in \Q\[[\ve\]]$ and 
$\chi^\vir(X,V)=\chi^\vir (X,\widetilde V,0).$
 \end{Proposition}
\begin{proof} This follows by combining the virtual Riemann-Roch theorem with the virtual localization.
We will show: 
\begin{equation}\label{claim} 
p_*(\ch(\widetilde V)\cdot \td(T_{X}^\vir)\cap [X]^\vir)=\chi^\vir(X,V,\ve).
\end{equation}
This implies the Corollary:  $\ch(\widetilde V)\cdot \td(T_{X}^\vir)\cap [X]^\vir\in A_*(X)\[[\ve\]]$ and thus by \eqref{claim}  we have $\chi^\vir(X,V,\ve)\in \Q\[[\ve\]]$. Furthermore the virtual Riemann-Roch theorem gives
$$\chi^\vir(X,V)=\int_{[X^\vir]} \ch(V)\cdot \td(T_{X}^\vir)=
p_*(\ch(\widetilde V)\cdot \td(T_{X}^\vir)\cap [X]^\vir)|_{\ve=0}=\chi^\vir(X,\widetilde V,0).$$
By the localization formula of  \cite{GP}
we have
\begin{align*}
p_*(\ch(\widetilde V)&\cdot \td(T_{X}^\vir)\cap [X]^\vir)=\sum_{i} p_{i*}\big(\ch(\widetilde V_i) \td(T_{X}^\vir|_{X_i})/\Eu(N_i^\vir)\cap [X_i]^\vir\big)\\
&=\sum_i p_{i*} \big(\td(T_{X_i}^\vir) \ch\big(\widetilde V_i/\ch(\Lambda_{-1} (N_i^\vir)^\vee\big)\cap [X_i]^\vir\big)=\sum_{i} p_{i*}^\vir\big(\widetilde V_{i}/\Lambda_{-1}(N_{i}^\vir)^\vee\big).
\end{align*}
where  we have used  $\td(T_{X}^\vir|_{X_i})=\td(T_{X_i}^\vir)\td(N_i^\vir)$, and 
$\td(N_i^\vir)=\Eu(N_i^\vir)/\ch(\Lambda_{-1}(N_i^\vir)^\vee)$, and finally \eqref{intt}.
\end{proof}

\begin{Conjecture}[$K$-theoretic virtual localization]\label{Kconj}
 Let $\iota:\bigcup X_i\to X$ be the inclusion. Then 
$\oo_X^\vir=\sum_{i=1}^s\iota_*\big(\oo_{X_i}^\vir/\Lambda_{-1} (N^\vir_i)^\vee\big),$
in localized equivariant $K$-theory.
\end{Conjecture}
In the context of 
$DG$-schemes \conref{Kconj} has already been proven in 
\cite[Thm.~5.3.1]{CFK} and it should be possible to adapt their proof.

\begin{Corollary}
Under the assumptions of \thmref{loc} we have
\begin{enumerate}
\item For any $V\in K^0(X)$ we have (writing $\widetilde V_i:=\widetilde V|_{X_i}$
for $\widetilde V$ an equivariant lift of $V$):
$$\chi^\vir_{-y}(X,V)=\Big(\sum_{i} \chi_{-y}^{\vir}\left(X_i,\widetilde V_i\otimes \Lambda_{-y}(N^\vir_i)^\vee/\Lambda_{-1}(N^\vir_i)^\vee\right)\Big)|_{\ve=0}.$$
\item Let $n_i:=\rk(N_i^\vir)$. Then 
$$Ell^\vir(X;z,\tau)=\Big(\sum_{i} y^{-n_i/2} Ell^\vir\left(X_i,\E(N_i^\vir)\Lambda_{-y}(N_i^\vir)/\Lambda_{-1}(N_i^\vir)^\vee,z,\tau\right)\Big)|_{\ve=0}.$$
\item $e^\vir(X)=\sum_{i} e^{\vir}(X_i)$, where $e^\vir(X_i)$ is defined using 
the obstruction theory induced from $X$.
In particular, if  all the $X_i$ are smooth and the obstruction theory induced from $X$ on each $X_i$ is the cotangent bundle, then $e^\vir(X)=e(X)$.
\end{enumerate}
\end{Corollary}
\begin{proof}
(1) By definition and \propref{loc} we get 
$$\chi_{-y}^\vir(X,V)=\chi^\vir(X,V\otimes \Lambda_{-y}(\Omega_X^\vir))=\sum_{i} p_{i*}^\vir\left(
V_i\otimes \Lambda_{-y}(\Omega_X^\vir|_{X_i})/\Lambda_{-1}(N_i^\vir)^\vee\right)|_{\ve =0}.$$
As $T_X^\vir=T_{X_i}^\vir+N_i^\vir\in K^0(X_i)$, and $\Lambda_{-y}$ is a homomorphism  we get 
\begin{align*}p_{i*}^\vir\left(V_i\otimes \Lambda_{-y}(\Omega_X^\vir|_{X_i})/\Lambda_{-1}(N_i^\vir)^\vee\right)&=
p_{i*}^\vir\left(V_i\otimes \Lambda_{-y}\Omega_{X_i}^\vir \otimes \Lambda_{-y}(N_i^\vir)^\vee/\Lambda_{-1}(N_i^\vir)^\vee\right)\\
&=\chi_{-y}^\vir \left(X_i,V_i\otimes \Lambda_{-y}(N_i^\vir)^\vee /\Lambda_{-1}(N_i^\vir)^\vee\right).\end{align*}
(2) Putting $y=e^{2\pi i z}$, $q=e^{2\pi i \tau}$, we have by definition and applying (1)
\begin{align*}
Ell^\vir(X,z,\tau)&=y^{-d/2}\chi^\vir_{-y}(X,\E(T_X^\vir))\\
&=y^{-d/2}\Big(\sum_{i} \chi_{-y}^{\vir}\left(X_i,\E(T_X^\vir) \otimes \Lambda_{-y}(N^\vir_i)^\vee/\Lambda_{-1}(N^\vir_i)^\vee\right)\Big)|_{\ve=0}\\
&=y^{-d/2}\Big(\sum_{i} \chi_{-y}^{\vir}\left(X_i,\E(T_{X_i}^\vir) \otimes \E(N^\vir_i) \otimes \Lambda_{-y}(N^\vir_i)^\vee/\Lambda_{-1}(N^\vir_i)^\vee\right)\Big)|_{\ve=0}\\&=\Big(\sum_{i} y^{-n_i/2} Ell^\vir\left(X_i,\E(N_i^\vir)\Lambda_{-y}(N_i^\vir)/\Lambda_{-1}(N_i^\vir)^\vee,z,\tau\right)\Big)|_{\ve=0}.
\end{align*}
(3) By the definition of $e^\vir(X)$ and by (1) we have 
\begin{align*}e^\vir(X)&=\chi^\vir_{-1}(X)=\sum_{i} \chi_{-1}^{\vir}\left(X_i,\Lambda_{-1}(N^\vir_i)^\vee/\Lambda_{-1}(N^\vir_i)^\vee\right)|_{\ve =0}\\&=\sum_{i} \chi_{-1}^{\vir}\left(X_i\right)=\sum_ie^\vir(X_i).\end{align*}
If all the $X_i$ are smooth with trivial obstruction theory, then    $e^\vir(X_i)=e(X_i)$, thus $e^\vir(X)=\sum_ie(X_i)$.
Put $Y:=X\setminus \bigcup_i X_i$. Then the $\C^*$ acts freely on $Y$, thus 
$e(Y)=0$. Furthermore $Y$ is open in $X$. Thus $e(X)=e(X\setminus Y)+e(Y)=\sum_ie(X_i)$.
\end{proof}

\section{Virtually smooth DM stacks}

Obstruction theories arise naturally as deformation-theoretic dimensional estimates on moduli spaces, e.g. of stable maps or stable sheaves. 
Such moduli spaces are often not schemes but DM stacks. In this section we will extend this paper's definition to the stack case, and discuss which results are still valid. A significant case where the theory works are moduli spaces of stable sheaves: in the forthcoming paper \cite{GMNY}, this will be used together with the results of  \cite{Moc} to study the invariants of  moduli spaces of sheaves on surfaces.

\subsection{Notation and conventions} 

A {\em DM stack} will be a separated algebraic stack in the sense of Deligne--Mumford of finite type over the ground field of characteristic $0$. All DM stacks will be assumed to be quasiprojective in the sense of Kresch, 
i.e. they admit  locally closed embeddings into a smooth DM stack which is proper over the base field and has projective coarse moduli space. Furthermore
we assume that they have  the resolution property;
in characteristic $0$ this last condition is implied by quasiprojectivity
 \cite[Thm.~5.3]{Kre}.
 
A DM stack $\mathcal X$ will be called {\em virtually smooth} of dimension $d$ if we have chosen a perfect obstruction theory $\phi:E^\bullet\to \tau_{\ge-1}L_\XX^\bullet$. 

The definitions for schemes are valid also in this case, yielding a virtual fundamental class $[\XX]^\vir\in A^d(\XX)$, a virtual structure sheaf $\OO_\XX^\vir\in K_0(\XX)$ and a virtual tangent bundle $T_\XX^\vir\in K^0(\XX)$.
  
If $\XX$ is proper and $V\in K^0(\XX)$ we can define 
$$\chi^\vir(\XX,V):=\chi(\XX,V\otimes \oo_{\XX}^\vir).$$

Following \defref{chiy}, we put
$\Omega_{\XX}^\vir:=(T_{\XX}^\vir)^\vee$, and define
the virtual $\chi_{-y}$-genus of $\XX$ by 
$\chi_{-y}^\vir(\XX):=\chi^\vir(\XX,\Lambda_{-y}\Omega_{\XX}^\vir)$ and, for 
$V\in K^0(\XX)$, put $\chi_{-y}^\vir(\XX,V):=\chi^\vir(\XX,V\otimes \Lambda_{-y}\Omega_{\XX}^\vir)$.
If $\chi_{-y}^\vir(\XX)$ is a polynomial, we define the virtual Euler characteristic of $\XX$ by  $e^\vir(\XX):=\chi_{-1}^\vir(\XX)$. Following \eqref{elldef}
we put $Ell^\vir(\XX,y,q):=y^{-d/2} \chi_{-y}^\vir(\XX,\E(T_{\XX}^\vir)).$

\subsection{Morphisms of virtually smooth DM stacks}

A morphism between virtually smooth DM stacks $(\XX,E_{\XX})$ and 
$(\YY,E_{\YY})$ is a pair $(f,\psi)$ where $f:\XX\to\YY$ is a morphism, and $\psi:f^*E_\YY\to E_\XX$ is a morphism in $D^b(\XX)$ such  that $$f^*\phi_\YY\circ \psi=\tau_{\ge-1}L_f\circ \phi_\XX:f^*E_\YY\to \tau_{\ge-1}L_\XX$$
and such that the mapping cone $C(\psi)$ is perfect in $[-1,0]$. 

Note that every fibre of a morphism of virtually smooth stacks is itself virtually smooth

A morphism $(f,\psi)$ of virtually smooth DM stacks is called  \'etale  if $f$ is \'etale  and $\psi$ is an isomorphism. 

\begin{Lemma}
Let $f:\XX\to\YY$ be an \'etale morphism of virtually smooth DM stacks.
\begin{enumerate} 
\item $T_{\XX}^\vir=f^*(T_\YY^\vir)$.
\item $[\XX]^\vir=f^*[\YY]^\vir$.
\item $\oo_{\XX}^\vir=f^*(\oo_{\YY}^\vir)$.
\end{enumerate}
\end{Lemma}
\begin{proof}
(1) follows  directly from the definitions.

The isomorphism $f^*L_{\YY}\to L_\XX$ induces an isomorphism of abelian cone stacks ${\frak N}_{\XX}\to {\frak N}_\YY$, and, by the consequence of 
\cite[Cor.~3.9]{BF}, an isomorphism 
${\frak C}_\XX\to f^*{\frak C}_\YY$. 
Choose a global resolution $[E^{-1}\to E^{0}]$
of $E_{\YY}$ and let $[F^{-1}\to F^0]$ be its pullback to $\XX$.
The cartesian diagram 
\begin{diagram}
\Cc_\XX&\rTo&[F_1/F_0]&\rTo &\XX\\
\dTo&&\dTo&&\dTo\\
\Cc_\YY&\rTo&[E_1/E_0]&\rTo&\YY
\end{diagram}
induces a cartesian diagram 
\begin{diagram}
C_\XX&\rTo&F_1&\rTo &\XX\\
\dTo&&\dTo_{\overline  f}&&\dTo_{f}\\
C_\YY&\rTo&E_1&\rTo&\YY
\end{diagram}
where $C_\XX$ (resp.~$C_\YY$) is the inverse image of $\Cc_\XX$
(resp. $\Cc_\YY$). Let $s_0$ (resp.~$\widetilde s_0$) 
denote the zero section of 
$F_1$ (resp.~$E_1$). Since $[\YY]^\vir=\widetilde s_0^![C_\YY]$ and 
$\overline f^*[C_\YY]=[C_{\XX}]$, (1) becomes 
$s_0^!\overline f^*[C_\YY]=f^*\widetilde s_0^![C_\YY]$, which is \cite[Thm.~6.2]{Fu}.
On the other hand
$$\Tor^{\oo_{F_1}}_k(\oo_{C_{\XX}},\oo_\XX)=\Tor^{\oo_{F_1}}_k(\overline f^*\oo_{C_{\YY}},\overline f^*\oo_\YY)=f^*(\Tor^{\oo_{E_1}}_k(\oo_{C_{\YY}},\oo_\YY)),
$$
since $\overline f$ is flat and hence commutes with $\Tor$; this proves (2).

\end{proof}
\subsection{The gerbe case}

Let $X$ be a scheme, and $\eps:\XX\to X$ a gerbe over $X$ banded by a finite abelian group $G$; note that $\eps$ is always an \'etale proper morphism. Write $|G|$ for the order of $G$.

For all $W\in K^0(X)$ we have  $\eps_*\eps^*(W)=W$.
The morphism $\eps^*:A^*(X)\to A^*(\XX)$ is a ring isomorphism and
$\eps_*\eps^*:A_*(X)\to A_*(X)$ is multiplication by $\frac{1}{|G|}$.
In particular for any class $\alpha\in A^*(X)$ we have  
$\int_{[X]^\vir}\alpha=|G|\int_{[\XX]^\vir} \eps^*(\alpha)$.

Let $\FF$  be a coherent sheaf on $\XX$; then for every point $x:\spec \C\to \XX$ the fiber $\FF(x):=x^*\FF$ is naturally a representation of $G$, and hence decomposes as a direct sum over the group $G^\vee$ of characters of $G$.
The fiberwise direct sum decompositions induce a global decomposition $E:=\sum E_\chi$ where the sum runs over $\chi\in G^\vee$, and any morphism $f:E^{-1}\to E^0$ of coherent sheaves respects the characters.

We write $\chi_0$ for the  trivial character. For $E$ a coherent sheaf on $\XX$, one has $\eps^*\eps_*E=E_{\chi_0}$; hence  $E$  is a pullback from $X$ if and only if 
$E=E_{\chi_0}$, that is $E_\chi=0$ for every $\chi\ne\chi_0$.

Let $\phi:E^\bullet\to \tau_{\ge -1}L_{\XX}$ be a $1$-perfect obstruction theory for $\XX$. The morphism $\eps$ is \'etale, hence the natural map $\eps^*L_X\to L_\XX$ is an isomorphism.

\begin{Remark} \label{pulob}The following are equivalent:
\begin{enumerate}
\item there is a (unique up to isomorphism) obstruction theory $F^\bullet$ on $X$ such that $\eps$ induces an \'etale morphism of virtually smooth stacks;
\item  for every  point $x:\spec \C\to \XX$ the $G$-representations $h^i((x^*E^\bullet)^\vee)$ are trivial, for $i=0,1$.
\end{enumerate}
\end{Remark}

For the rest of this subsection we assume 
that the conditions of \remref{pulob} are fulfilled.
As before  let $E_{0}\to E_{1}$ be the dual complex to $E^\bullet$ and $F_{0}\to F_{1}$ the dual complex to $F^\bullet$, and let $T_{\XX}^\vir:=E_{0}-E_{1}$ and $T_{X}^\vir:=F_{0}-F_{1}$.
Let $[\XX]^\vir\in A^d(\XX)$, (resp.~$[X]^\vir\in A^d(X)$) be  the virtual fundamental classes defined via $E^\bullet$ (resp.~$F^\bullet$).
Denote $\oo_{\XX}^\vir\in K^0(\XX)$, (resp.~$\oo_X^\vir\in K^0(X)$) the virtual structure sheaves
on $\XX$ (resp.~$X$) defined via $E^\bullet$ (resp.~$F^\bullet$).
As before  for any $V\in K^0(\XX)$ and any $W\in K^0(X)$ let 
$$\chi^\vir(\XX,V):=\chi(\XX,V\otimes \oo_{\XX}^\vir),\quad 
\chi^\vir(X,W):=\chi(X,W\otimes \oo_{X}^\vir).$$

\begin{Corollary}\label{pull}
\begin{enumerate}

\item For any $V\in K^0(\XX)$ we have 
$\chi^\vir(\XX,V)=\chi^\vir(X,\eps_*(V)).$
\item For any $W\in K^0(X)$, we have 
$\chi^\vir(\XX,\eps^*(W))=\chi^\vir(X,W)$.
\end{enumerate}
\end{Corollary}
\begin{proof} 
(1) By the projection formula 
$$\chi^\vir(\XX,V)=\chi(\XX,V\otimes \oo_{\XX}^\vir))=
\chi(X,\eps_*(V)\otimes \oo_{X}^\vir)=\chi^\vir(X,\eps_*(V)).$$
(2) By (1)  $\chi^\vir(\XX,\eps^*(W))=\chi^\vir(X,\eps_*(\eps^*(W)))=
\chi^\vir(X,W)$.
\end{proof}
\begin{Corollary}\label{SRR}
\begin{enumerate}
\item Let $V\in K^0(\XX)$, then $$\chi^\vir(\XX,V)=|G|\int_{[\XX]^\vir} \ch(\eps^*\eps_*(V))\td(T_{\XX}^\vir).$$
\item Let $W\in K^0(X)$, then 
$$\chi^\vir(\XX,\eps^*W)=
|G|\int_{[\XX]^\vir} \ch(\eps^*(W))\td(T_{\XX}^\vir).$$
\end{enumerate}
\end{Corollary}
\begin{proof}
(1) By \corref{pull} and virtual Riemann-Roch we get 
$$\chi^\vir(\XX,V)=\chi^\vir(X,\eps_*(V))=\int_{[X]^\vir} \ch(\eps_*(V))\td(T_X^\vir)=
|G|\int_{[\XX]^\vir} \eps^*(\ch(\eps_*(V)))\td(T_{\XX}^\vir),$$ and the claim follows because $\ch$ commutes with pullback.
(2) Again  \corref{pull} and virtual Riemann-Roch give
$$\chi^\vir(\XX,\eps^*W)=\chi^\vir(X,W)=\int_{[X]^\vir} \ch(W)\td(T_X^\vir)=
|G|\int_{[\XX]^\vir} \ch(\eps^*(W))\td(T_{\XX}^\vir).$$
\end{proof}

\begin{Corollary}
\begin{enumerate}
\item $\chi_{-y}^\vir(\XX)=\chi_{-y}^\vir(X)$, $e^\vir(\XX)=e^\vir(X)$.
\item More generally for $V\in K^0(\XX)$ and $W\in K^0(X)$, we have 
$\chi_{-y}^\vir(\XX,V)=\chi_{-y}^\vir(X,\eps_*(V))$ and 
$\chi_{-y}^\vir(\XX,\eps^*W)=\chi_{-y}^\vir(X,W)$.
\item $Ell^\vir(\XX,y,q)=Ell^\vir(X,y,q)$.
\item $e^\vir(X)=|G|\int_{[\XX]^\vir} c_d(T_{\XX}^\vir).$
\end{enumerate}
\end{Corollary}
\begin{proof}
(1), (2), (3) follow immediately from the  definitions and \corref{pull}.
(4) follows from \corref{actually}.
\end{proof}

\subsection{Moduli stacks of stable sheaves}

Let $V$ be a projective variety of dimension $d$, $H$ an ample line bundle on $V$, $r>0$ an integer and (for $i=1,\ldots,r$) $c_i\in H^{2i}(X,\ZZ)$ cohomology classes. We denote by $\Xx$ the moduli stack of $H$--stable bundles of rank $r$ and Chern classes $c_i$ on $V$. We denote as usual by $\Pic V$ the Picard group of $V$, and by  $\Ppic V$ the Picard stack of $V$, i.e. the moduli stack of line bundles on $V$. The determinant defines a natural morphism $\det:\Xx\to\Ppic V$, and there is a natural map $\Ppic V\to \Pic V$ which identifies $\Pic V$ with the coarse moduli space of $\Ppic V$.

Both $\Xx$ and $\Ppic V$ are algebraic stacks in the sense of Artin, and they have a natural structure of gerbes banded by $\Gg_m$ over their coarse moduli spaces, since all stable bundles and all line bundles are simple, i.e. their automorphism group is given by nonzero scalar multiples of the identity. It is a well known fact that the gerbe $\Ppic V\to \Pic V$ is trivial, and we choose a trivialization i.e., a Poincar\'e line bundle on $V\times \Pic V$, that is a section of the structure morphism $\Ppic X\to \Pic X$. We denote by $\XX$ the fiber product $\Xx\times_{\Ppic V}\Pic V$; if $L\in \Pic V$ is a line bundle, we denote by $\XX_L$ the fiber of $\XX$ over $L$. Both $\XX$ and $\XX_L$ are DM stacks, and are naturally gerbes banded by $\mu_r$ over their coarse moduli spaces, which we denote by $X$ and $X_L$ respectively. Both $X$ and $X_L$ are quasiprojective schemes; they are projective if there are no strictly semistable sheaves with the given rank and Chern classes (and determinant, in the case of $X_L$).

Let $\FF$ be an $H$-stable sheaf of rank $r$, with Chern classes $c_i$ and determinant $L$, and denote by $f$ the corresponding morphism from $\spec \mathbb C$ to either $\XX$ or $\XX_L$. The stack $\XX$ and $\XX_L$ have natural obstruction theories $E^\bullet$ and $E_L^\bullet$ with the property that $h^i(f^*(E_L^\bullet)^\vee)=\Ext^{i+1}_0(\FF,\FF)$, $h^1(f^*(E^\bullet)^\vee)=\Ext^2_0(\FF,\FF)$, $h^0(f^*(E_L^\bullet)^\vee)=\Ext^1(\FF,\FF)$. 

\begin{Lemma} Let $c\in \mathbb C$ be a nonzero scalar. Then the automorphism induced on $\Ext^i(\FF,\FF)$ and $\Ext^i_0(\FF,\FF)$ by acting simultaneously on both copies of $\FF$ with the scalar $c$ is the identity.
\end{Lemma}
\begin{proof} Since $\Ext$ is contravariant in the first variable and covariant in the second, the scalar automorphism $c$ applied to the first variable acts as $c^{-1}$, and applied to the second variable it acts as $c$. The vector space $\Ext^i_0(\FF,\FF)$ carries the induced action.
\end{proof}
\begin{Proposition} The complex $E^\bullet$ (respectively $E^\bullet_L$) is the pullback of a unique obstruction theory  on $X$ (resp.~on $X_L$).
\end{Proposition}
\begin{proof} This follows immediately by \remref{pulob}.
\end{proof}

\section*{Appendix -- A detailed proof of Theorem \ref{Euler}}

Fix formal variables $x_i$ for $i=1,\ldots,n$ and $y_j$ for $j=1,\ldots,m$ and let $A$ be the formal power series ring $A:=\Q\[[x_i,y_j\]]$. Let $d=n-m$ which we assume nonegative, $\m_A$ the maximal ideal in $A$, and let $\bar A$ be the quotient ring $A/\m_A^{d+1}$ (formal power series developments up to order $d$). 

Let $t$ indicate one of the variables $x_i$'s or $y_j$'s.
Let $$f(t,y):=\frac{t}{1-e^{-t}}(1-ye^{-t})\in A[y].$$ Since $t/(1-e^{-t})$ is invertible in $A$, $f$ is invertible in $A\[[y\]]$; its inverse is the power series $$
f_*(t,y):=\frac{1-e^{-t}}{t}(\sum_{r\ge 0}y^re^{rt})\in A\[[y\]].$$

We now consider the variable change $y=1-1/u$, i.e. $u=1/(1-y)$. Note that $g(t,u):=uf(t,1-1/u)\in A[u,1/u]$ is actually a degree one polynomial, which can be written as $$g(t,u)=ut+\frac{te^{-t}}{1-e^{-t}}=1+\ell+tu$$ with $\ell\in \m_A$. Note that $g$ is invertible and its inverse is $$g_*(t,u)=\sum_{r\ge 0}(-1)^r(\ell+tu)^r.$$

Let $\xi=\sum \xi_ku^k\in A\[[u\]]$ be a power series. We say that $\xi$ is good if $\xi_k\in \m_A^k$ for every $k$; in that case, we denote by $\tilde \xi$ the power series $\sum \tilde\xi_ku^k$ where $\tilde\xi_k$ is the homogeneous part of degree $k$ in $\xi_k$. Good power series are closed under sum, product, and infinite sum when this makes sense, and the map $\xi\mapsto \tilde\xi$ commutes with all these operations.
\begin{Claim} The power series $g_*(t,u)$ is good, and $\tilde g_*(t,u)=(1+tu)^{-1}$.
\end{Claim}
\begin{proof} Let $\xi:=\ell+tu$. Then $\xi$ is good and $\tilde \xi=tu$. Hence $g_*$ is good, and $\tilde g_*:=\sum_{r\ge 0}(-1)^r\tilde\xi^r=\sum_{r\ge 0}(-1)^r (tu)^r=(1+tu)^{-1}$.
\end{proof}
We denote by $\bar f$ the image of $f$ in $\bar A[u]$, and similarly for $f_*,g,g_*$. Note that by the claim $\bar g_*$ is actually a polynomial in $u$ of degree at most $d$.

Since $g$ is a polynomial, we can consider $g(t,1/(1-y))\in A\[[y\]]$; it is easy to check that $g(t,1/(1-y))=(1-y)^{-1}f(t,y)$. In particular $\bar g(t,1/(1-y))=(1-y)^{-1}\bar f(t,y)$.

Since $\bar g_*$ is a polynomial in $u$, it makes sense to consider $\bar g_*(t,1/(1-y))\in \bar A\[[y\]]$.

\begin{Claim} In the ring $\bar A\[[y\]]$ one has the equality $$
\frac{1}{1-y}\bar g_*(t,1/(1-y))=\bar f_*(t,y).$$
\end{Claim}
\begin{proof} From the equations $g g_*=1$ in $A\[[u\]]$ and $f f_*=1$ in $A\[[y\]]$ it follows that $\bar f\bar f_*=1$ and $\bar g(t,1/(1-y))\bar g_*(t,1/(1-y))=1$ in $\bar A\[[y\]]$. We can multiply both sides of the claim  by $\bar g(t,1/(1-y))$ and $\bar f(t,y)$ since these are invertible elements. The result to prove becomes $$
\frac{1}{1-y}\bar f(t,y)= \bar g(t,1/(1-y))$$
which immediately follows from the corresponding equality in $A\[[y\]]$.\end{proof}

The claim above would not make sense in $A\[[y\]]$ since we cannot substitute in the power series $g_*$ the power series $(1-y)^{-1}$ which has a nonzero constant term.

Let $$\XX_{-y}(X):=\prod_{i=1}^n \bar f(x_i,y)\prod _{j=1}^m \bar f_*(y_j,y)\in \bar A\[[y\]].$$

\begin{Claim} (1) In the ring $\bar A\[[y\]]$ one has the equality
$$\XX_{-y}(X)=(1-y)^{d}\prod_{i=1}^n \bar g(x_i,1/(1-y))\prod_{j=1}^m\bar g_*(y_j,1/(1-y))$$
(2) $\XX_{-y}(X)$ is a polynomial of degree at most $d$.\\
(3) Write $\XX_{-y}(X)=\sum_{l=0}^r \XX^l(1-y)^l$. Then $\XX^l-c_l(T_X^\vir)\in \m_{\bar A}^{l+1}$. 
\end{Claim}

\begin{proof} (1) We can apply the previous claim since $(1-y)^d=(1-y)^n/(1-y)^m$.\\
(2) The power series $$h(u):=\prod_{i-1}^n g(x_i,u)\cdot\prod_{i=1}^m g_*(y_j,u)\in A\[[u\]]$$ is good since each of its factors is good, and therefore $\bar h(u)$ is a polynomial of degree at most $d$.\\
(3) It is enough to prove that $\tilde h(u)=\prod_{i=1}^n(1+x_iu)\prod_{j=1}^m(1+y_ju)^{-1}$, and this follows from the definition and the first claim.
\end{proof}

\end{document}